\documentclass[final,5p,times,twocolumn]{elsarticle}

\usepackage{epsfig}
\usepackage{amssymb}
\usepackage{amsthm}
\usepackage{lineno}
\usepackage{amsmath,amsfonts}

\usepackage{algorithmic}
\usepackage{algorithm}
\usepackage{array}
\usepackage[caption=false,font=normalsize,labelfont=sf,textfont=sf]{subfig}
\usepackage{textcomp}
\usepackage{stfloats}
\usepackage{url}
\usepackage{verbatim}
\usepackage{graphicx}
\usepackage[justification=centering]{caption}
\usepackage{multirow}
\usepackage{makecell}
\usepackage{color}
\usepackage{verbatim}
\usepackage{appendix}

\biboptions{sort&compress} 

\makeatletter 
\renewcommand{\maketag@@@}[1]{\hbox{\m@th\normalsize\normalfont#1}}%
\makeatother

\newcommand{\KP}[1]{{\color{blue} #1}}

\journal{European Journal of Control}

\begin{document}

\newtheorem{remark}{Remark}
\newtheorem{theorem}{Theorem}
\newtheorem{corollary}{Corollary}
\newtheorem{lemma}{Lemma}
\newtheorem{proof}{Proof}
\newtheorem{assumption}{Assumption}

\begin{frontmatter}



\title{Distributed online constrained convex optimization \\with event-triggered communication}

 \author[label1]{Kunpeng Zhang}
 \ead{2110343@stu.neu.edu.cn}
 \affiliation[label1]{organization={State Key Laboratory of Synthetical Automation for Process Industries, Northeastern University},
             city={Shenyang},
             postcode={110819},
             country={China}}

  \author[label2]{Xinlei Yi}
   \ead{xinleiyi@mit.edu}
 \affiliation[label2]{organization={Laboratory for Information and Decision Systems, Massachusetts Institute of Technology},
             city={Cambridge},
             postcode={MA 02139},
             country={USA}}

 \author[label1]{Yuzhe Li}
\ead{yuzheli@mail.neu.edu.cn}

 \author[label3]{Ming Cao}
 \ead{m.cao@rug.nl}
 \affiliation[label3]{organization={Engineering and Technology Institute Groningen, University of Groningen},
             addressline={AG 9747},
             city={Groningen},
             country={The Netherlands}}

\author[label1]{Tianyou Chai}
\ead{tychai@mail.neu.edu.cn}

 \author[label1]{Tao Yang\corref{cor1}}
 \ead{yangtao@mail.neu.edu.cn}
 \cortext[cor1]{Corresponding author.}


\begin{abstract}
This paper focuses on the distributed online convex optimization problem with time-varying inequality constraints over a network of agents, where each agent collaborates with its neighboring agents to minimize the cumulative network-wide loss over time. To reduce communication overhead between the agents, we propose a distributed event-triggered online primal--dual algorithm over a time-varying directed graph.
With several classes of appropriately chose decreasing parameter sequences and non-increasing event-triggered threshold sequences, we establish dynamic network regret and network cumulative constraint violation bounds.
Finally, a numerical simulation example is provided to verify the theoretical results.
\end{abstract}

\begin{keyword}
Cumulative constraint violation \sep event-triggered communication \sep inequality constraints \sep online convex optimization
\end{keyword}

\end{frontmatter}


\section{Introduction}
Distributed optimization has wide applications in sensor networks \cite{Rabbat2004}, machine learning \cite{Nedic2020} and power systems \cite{Molzahn2017}, where a network of agents aims at minimizing the average of all the local cost functions by exchanging local information of the agents. Distributed optimization can be traced back at least to \cite{Tsitsiklis1986, Bertsekas1989}, and the past decades have witnessed its rapid development, see survey papers \cite{Cao2012, Nedic2018, Yang2019}. However, the local loss functions are static, which may not be applicable to dynamic and uncertain environments \cite{Xiong2020}.

Distributed online convex optimization is a promising framework due to its powerful modeling capability for various problems in dynamic, uncertain and even adversarial environments. In distributed online convex optimization, the agents collaboratively make decisions without knowing their local loss functions at the current iteration, and then the time-varying local loss functions are privately revealed. The goal is to minimize the cumulative network-wide loss over time. In general, static regret is the standard performance metric to evaluate online algorithms, which measures the difference of the cumulative loss between the decision sequence and the optimal static decision in hindsight.
Various distributed online algorithms with sublinear static regret have been developed, see \cite{Yan2012, Shahrampour2017, Lee2017, Zhang2017, Yuan2017, Yuan2021b, Yuan2022, Yi2020, Li2020, Li2022}, recent survey paper \cite{Li2023} and references therein.
For example, the authors of \cite{Yan2012} develop a projection-based distributed online subgradient descent algorithm by using Bregman divergence in lieu of Euclidean distance for projection, the authors of \cite{Shahrampour2017} propose a decentralized online mirror descent algorithm.

Note that the aforementioned distributed online algorithms require local information exchange between the agents via the underlying communication network at each iteration, which may cause large amount of communication overhead.
To overcome this limitation, by equipping event-triggered communication scheme to the algorithm proposed in \cite{Yan2012}, the authors of \cite{Cao2021} propose two distributed event-triggered algorithms with full-information feedback and bandit feedback over a fixed undirected graph, respectively.
Moreover, sublinear static regret is achieved for both algorithms when their event-triggering threshold sequences are non-increasing and converge to zero.
By using one-point and two-point subgradient estimators respectively, two distributed event-triggered algorithms with bandit feedback are developed in \cite{Xiong2023} for the fix delayed bandit feedback case and sublinear static regret is established for the algorithms.
The authors of \cite{Paul2022} develop a distributed event-triggered algorithm based on the algorithm proposed \cite{Shahrampour2017} and sublinear static regret is achieved.

It is worth mentioning that the above algorithms which achieve performance close to the best static regret may perform poorly in terms of dynamic regret \cite{Besbes2015}.
Dynamic regret is a more stringent performance metric, which measures the difference of the cumulative loss between the decision sequence and the best decision sequence selected by a clairvoyant that knows the sequence of loss functions in advance.
By using the first and second moments of the gradient of the local loss functions, the authors of \cite{Oakamoto2023} develop the event-triggered algorithm with full-information feedback in \cite{Cao2021} by imposing adaptive updating step-sizes, and analyze dynamic regret.

Most existing studies for online constrained convex optimization focus on the case where the feasible set is a simple closed convex set (a box or a ball).
To cope with more complex scenarios, the authors of \cite{Yuan2018, Yi2021a} characterize the feasible set by inequality constraints and a simple closed convex set for centralized online constrained convex optimization.
Note that projection operation is needed at each iteration in order to satisfy always inequality constraints, which results in heavy computation burden.
Therefore, the authors treat these inequality constraints as long-term constraints, i.e., the inequality constraints are allowed to be violated but are satisfied in the long run.
Inspired by \cite{Yuan2018, Yi2021a}, a distributed online primal--dual algorithm with full-information feedback is proposed in \cite{Yi2023}, where agents need to share their local decisions with their neighboring agents via the underlying communication topology at each iteration.
To reduce communication overhead, in this paper, we propose a distributed event-triggered online primal--dual algorithm over an uniformly jointly strongly connected time-varying directed graph by integrating event-triggered communication into the algorithm in~\cite{Yi2023}, where each agent broadcasts its current local decision to its neighboring agents only if norm of the difference between the decision and the last broadcasted decision is not less than the current event-triggering threshold. In addition, base on several classes of appropriately chosen event-triggered threshold sequences, we analyzes the impact of event-triggering threshold on dynamic network regret and network cumulative constraint violation.

The contributions are as follows.
\vspace{-2mm}
\begin{itemize}
\item[$\bullet$]
To the best of our knowledge, this paper is among the first to consider time-varying inequality constraints for distributed online convex optimization with event-triggered communication.
Compared to distributed event-triggered online algorithms \cite{Cao2021, Xiong2023, Paul2022, Oakamoto2023}, which only consider a simple closed convex constrained set and an undirected fixed communication graph, we consider time-varying inequality constraints and a directed time-varying graph.

\vspace{-2mm}

\item[$\bullet$]
In Theorem~1, we show that the proposed algorithm achieves sublinear dynamic network regret and network cumulative constraint violation if the path--length of the benchmark, the accumulated dynamic variation of the optimal decision sequence, grows sublinearly and the non-increasing event-triggering threshold sequence converges to zero.
With two classes of natural decreasing event-triggering threshold sequences, in Corollaries~1 and~2, we respectively establish in sublinear dynamic network regret and network cumulative constraint violation bounds.
These dynamic network regret bounds recover the results achieved by the centralized online algorithm without event-triggered communication in \cite{Hall2015}, and the distributed event-triggered online algorithm in \cite{Oakamoto2023}.

\vspace{-2mm}

\item[$\bullet$]
In Theorem~2, by appropriately designing the parameter sequences,
we avoid the impact of the event-triggering threshold on the updating step-sizes of the local primal variables, and establish sublinear dynamic network regret and network cumulative constraint violation bounds provided that the path–length of the benchmark grows sublinearly.
These bounds recover the results achieved by the distributed online algorithms in \cite{Yi2023} without event-triggered communication.
\end{itemize}

The remainder of this paper is as follows. Section~II presents the problem formulation and motivation. Section~III proposes the distributed event-triggered online primal--dual algorithm and the performance metrics.
Section~IV analyzes the performance of the proposed algorithm.
Section~V demonstrates numerical simulations.
Finally, Section~VI concludes the paper.
All detailed proofs are provided in Appendix.

\textbf{Notations:} ${\mathbb{N}_ + }$, $\mathbb{R}$, ${\mathbb{R}^p}$ and $\mathbb{R}_ + ^p$ denote the set of all positive integers, real numbers, $p$-dimensional vectors and nonnegative vectors, respectively. $[ n ]$ denotes the set $\{ {1, \cdot  \cdot  \cdot ,n} \}$ for any $n \in {\mathbb{N}_ + }$. Given vectors $x$ and $y$, ${x^T}$ denotes the transpose of the vector $x$, and $\langle {x,y} \rangle $ and $x \otimes y$ denote the standard inner and Kronecker product of the vectors $x$ and $y$, respectively. ${\mathbf{0}_m}$ denotes the $m$-dimensional column vector whose components are all $0$. $\mathrm{col}( {q_1}, \cdot  \cdot  \cdot ,{q_n} )$ denotes the concatenated column vector of ${q_i} \in {\mathbb{R}^{{m_i}}}$ for $i \in [ n ]$. For a set $\mathbb{K} \in {\mathbb{R}^p}$, ${\mathcal{P}_{\mathbb{K}}}(  \cdot  )$ denotes a projection operator, i.e., ${\mathcal{P}_{\mathbb{K}}}( x ) = \arg {\min _{y \in {\mathbb{K}}}}{\| {x - y} \|^2}$, $\forall x \in {\mathbb{R}^p}$, and $[  \cdot  ]_+$ denotes ${\mathcal{P}_{\mathbb{R}_ + ^p}}( \cdot )$. For a scalar function $f:{\mathbb{R}^p} \to \mathbb{R}$, $\partial f( x )$ denotes the subgradient of $f$ at $x$.

\vspace{-3mm}

\section{Problem formulation and motivation}
In this section, we formulate the distributed online convex optimization problem with time-varying inequality constraints, and then present the motivation.

Consider a repeated game with $T$ iterations over a network of $n$ agents. At iteration $t$, the agents indexed by $i \in [ n ]$ exchange information with their neighboring agents via an underlying communication topology, and then select decisions ${x_{i,t}} \in \mathbb{X}$ without knowing the local loss functions ${f_{i,t}}:\mathbb{X} \to \mathbb{R}$ and constraint functions ${g_{i,t}}:\mathbb{X} \to {\mathbb{R}^{{m_i}}}$, where $\mathbb{X} \subseteq {\mathbb{R}^p}$ is a known convex set to the agents. After that, the local loss functions ${f_{i,t}}$ and constraint functions ${g_{i,t}}$ are privately revealed. Accordingly, the agents suffer losses ${f_{i,t}}( {{x_{i,t}}} )$. The goal of the network is to minimize the average of the network-wide loss accumulated over $T$ iterations, i.e.,  $\frac{1}{n}\sum\nolimits_{i = 1}^n {\sum\nolimits_{t = 1}^T {{f_t}\left( {{x_{i,t}}} \right)} } $.
Note that at iteration~$t$,
\begin{flalign}
{f_t}( x ) &= \frac{1}{n}\sum\limits_{j = 1}^n {{f_{j,t}}( x )},  \label{pf-eq1} \\
{g_t}( x ) &= {\rm{col}}\big( {{g_{1,t}}( x ), \cdot  \cdot  \cdot ,{g_{n,t}}( x )} \big), \label{pf-eq2}
\end{flalign}
are the global loss and constraint functions of the network, respectively.

Let the set of the feasible decision sequences
\begin{flalign}
{\mathcal{X}_T} = \{ {( {{x_1}, \cdot  \cdot  \cdot ,{x_T}} ):{g_t}( {{x_t}} ) \le {\mathbf{0}_m},{x_t} \in \mathbb{X},\forall t \in [ T ]} \}, \label{pf-eq3}
\end{flalign}
and the set of the feasible static decision sequences
\begin{flalign}
{\mathcal{\hat{X}}_T} = \{ {( {x, \cdot  \cdot  \cdot ,x} ):{g_t}( x ) \le {\mathbf{0}_m},x \in \mathbb{X},\forall t \in [ T ]} \}, \label{pf-eq4}
\end{flalign}
are non-empty, where $m = \sum\nolimits_{j = 1}^n {{m_j}}$.

The communication topology among agents is described by a time-varying directed graph ${\mathcal{G}_t} = ( {\mathcal{V},{\mathcal{E}_t}} )$, where $\mathcal{V} = [ n ]$ is the set of agents and ${\mathcal{E}_t} \subseteq \mathcal{V} \times \mathcal{V}$ is the set of edges at iteration~$t$. A directed edge $( {j,i} ) \in {\mathcal{E}_t}$ implies that agent $i$ can receive information from agent $j$ at iteration~$t$. The sets of in- and out-neighbors of agent $i$ at iteration~$t$ are $\mathcal{N}_i^{\text{in}}( {{\mathcal{G}_t}} ) = \{ {j \in [ n ]|( {j,i} ) \in {\mathcal{E}_t}} \}$ and $\mathcal{N}_i^{\text{out}}( {{\mathcal{G}_t}} ) = \{ {j \in [ n ]|( {i,j} ) \in {\mathcal{E}_t}} \}$, respectively.
The associated weight mixing matrix ${W_t} \in {\mathbb{R}^{n \times n}}$ satisfies that ${[ {{W_t}} ]_{ij}} > 0$ if $( {j,i} ) \in {\mathcal{E}_t}$ or $i = j$, and ${[ {{W_t}} ]_{ij}} = 0$ otherwise.

In this paper, the following assumptions are made, which are commonly adopted in distributed optimization, see, e.g., \cite{Cao2021, Xiong2023, Paul2022, Oakamoto2023, Yi2023, Mohammadreza2024}, recent survey paper \cite{Li2023} and references therein.
\begin{assumption}
(i) The set $\mathbb{X}$ is convex and closed.
Moreover, it is bounded by a positive constant $R( \mathbb{X} )$, i.e., for any $x \in \mathbb{X}$
\begin{flalign}
\| x \| \le R( \mathbb{X}). \label{ass-eq1}
\end{flalign}
(ii) For all $i \in [n]$, $t \in {\mathbb{N}_ + }$, the local loss functions ${f_{i,t}}$ and constraint functions ${g_{i,t}}$ are convex, and there exists a positive constant ${F_1}$ such that
\begin{subequations}
\begin{flalign}
| {{f_{i,t}}( x ) - {f_{i,t}}( y )} | &\le {F_1}, \label{ass-eq2a} \\
\| {{g_{i,t}}( x )} \| &\le {F_1}, x, y \in \mathbb{X}. \label{ass-eq2b}
\end{flalign}
\end{subequations}
(iii) For all $i \in [n]$, $t \in {\mathbb{N}_ + }$, the subgradients $\partial {f_{i,t}}( x )$ and $\partial {g_{i,t}}( x )$ exist, and there exists a positive constant ${F_2}$ such that
\begin{subequations}
\begin{flalign}
\| {\partial {f_{i,t}}( x )} \| &\le {F_2}, \label{ass-eq3a}\\
\| {\partial {g_{i,t}}( x )} \| &\le {F_2}, x \in \mathbb{X}. \label{ass-eq3b}
\end{flalign}
\end{subequations}
\end{assumption}
\begin{assumption}
For $t \in {\mathbb{N}_ + }$, the time-varying directed graph $\mathcal{G}_t$ satisfies that

\noindent (i) There exists a constant $w  \in ( {0,1} )$, such that ${[ {{W_t}} ]_{ij}} \ge w$ if ${[ {{W_t}} ]_{ij}} > 0$.

\noindent (ii) The mixing matrix ${W_t}$ is doubly stochastic, i.e., ${\sum\nolimits_{i = 1}^n {[ {{W_t}} ]} _{ij}} = {\sum\nolimits_{j = 1}^n {[ {{W_t}} ]} _{ij}} = 1$, $\forall i,j \in [ n ]$.

\noindent (iii) There exists an integer $B > 0$ such that the time-varying directed graph $( {\mathcal{V}, \cup _{l = 0}^{B - 1}{\mathcal{E} _{t + l}}} )$ is strongly connected.
\end{assumption}

Assumption~1 implies that the local loss functions ${f_{i,t}}$ and constraint functions ${g_{i,t}}$ are convex and Lipschitz continuous on $\mathbb{X}$.
Assumption~2 ensures that the time-varying directed graph~$\mathcal{G}_t$ is
uniformly jointly strongly connected, which is considerably weaker than requiring $\mathcal{G}_t$ to be a strongly connected graph as it allows that there exists a path from one agent to every other agent within any bounded interval of length $B$.

The authors of \cite{Yi2023} consider the distributed online convex optimization problem with time-varying inequality constraints, and propose a distributed online primal--dual algorithm with full-information feedback, where the agents at each iteration need to share their decisions through a communication network. However, network resources are often limited. To reduce communication overhead, this paper integrates event-triggered communication into the algorithm.

\section{Distributed event-triggered online primal--dual algorithm}
In this section, we propose a distributed event-triggered online primal--dual algorithm. Moreover, we present the performance metrics to evaluate the algorithm.

\subsection{Algorithm description}
By integrating event-triggered communication into the algorithm with full-information feedback in \cite{Yi2023}, the distributed event-triggered online primal--dual algorithm is proposed from the perspective of each agent, which is presented in pseudo-code as Algorithm 1.
\begin{algorithm}[t]
  \caption{Distributed Event-Triggered Online Primal--Dual Algorithm} 
  \begin{algorithmic}
  \renewcommand{\algorithmicrequire}{\textbf{Input:}}
  \REQUIRE
    Decreasing and positive sequences $\{ {\alpha _t}\} $, $\{ {\beta _t}\} $, $\{ {\gamma _t}\} $, and non-increasing
 and positive sequence $\{ {\tau _t}\} $.
  \renewcommand{\algorithmicrequire}{\textbf{Initialize:}}
  \REQUIRE
     For $i \in [ n ]$, initialize ${x_{i,1}} \in \mathbb{X}$, ${{\hat x}_{i,1}} = {x_{i,1}}$ and ${q_{i,1}} = {\mathbf{0}_{{m_i}}}$,  and broadcast ${{\hat x}_{i,1}}$ to $\mathcal{N}_i^{\text{out}}( {{\mathcal{G}_1}} )$ and receive ${{\hat x}_{j,1}}$ from $j \in \mathcal{N}_i^{\text{in}}( {{\mathcal{G}_1}} )$.
    \FOR {$t = 1, \cdot  \cdot  \cdot, T-1 $}
    \FOR {$i = 1,\cdot  \cdot  \cdot,n$ in parallel}
    \STATE Observe $\partial {f_{i,t}}({x_{i,t}})$, $\partial {[{g_{i,t}}({x_{i,t}})]_ + }$, and ${{{[{g_{i,t}}({x_{i,t}})]}_ + }}$;
    \STATE Distributed consensus protocol: \par\nobreak\vspace{-10pt}
     \begin{small}
     \begin{flalign}
       {z_{i,t + 1}} &= \sum\limits_{j = 1}^n {{{[{W_t}]}_{ij}}{{\hat x}_{j,t}}}, \label{Algorithm1-eq1}
     \end{flalign}
     \end{small}%
    \STATE Primal--dual protocol: \par\nobreak\vspace{-10pt}
     \begin{small}
     \begin{flalign}
       {\omega _{i,t + 1}} &= \partial {f_{i,t}}({x_{i,t}}) + \partial {[{g_{i,t}}({x_{i,t}})]_ + }{q_{i,t}}, \label{Algorithm1-eq2}\\
       {x_{i,t + 1}} &= {\mathcal{P}_\mathbb{X}}({z_{i,t + 1}} - {\alpha _{t + 1}}{\omega _{i,t + 1}}), \label{Algorithm1-eq3}\\
       \nonumber
       {q_{i,t + 1}} &= \Big[ ( {1 - {\beta _{t + 1}}{\gamma _{t + 1}}} ){q_{i,t}} + {\gamma _{t + 1}}\Big( {{{[{g_{i,t}}(x_{i,t})]}_ + }} \\
       & \;\;\;\;{ + {{\big( {\partial {{[ {{g_{i,t}}({x_{i,t}})} ]}_ + }} \big)}^T}( {{x_{i,t + 1}} - {x_{i,t}}} )} \Big) \Big]_ +. \label{Algorithm1-eq4}
      \end{flalign}
      \end{small}%
     \STATE Event-triggering check:
     \IF {$\| {{x_{i,t + 1}} - {{\hat x}_{i,t}}} \| \ge {\tau _{t + 1}}$}
     \STATE Set ${{\hat x}_{i,t + 1}} = {x_{i,t + 1}}$, and broadcast ${{{\hat x}_{i,t + 1}}}$ to $\mathcal{N}_i^{\text{out}}( {{\mathcal{G}_{t+1}}} )$.
     \ELSE
     \STATE Set ${{\hat x}_{i,t + 1}} = {{\hat x}_{i,t}}$, and do not broadcast.
     \ENDIF
    \ENDFOR
    \ENDFOR
  \renewcommand{\algorithmicensure}{\textbf{Output:}}
  \ENSURE
      $\{ x_{i,t} \}$.
  \end{algorithmic}
\end{algorithm}

In Algorithm 1, for $t \in [ T ]$ with $t \ge 2$ and ${i \in [ n ]}$, by the distributed consensus protocol \eqref{Algorithm1-eq1}, agent $i$ computes ${z_{i,t}} \in \mathbb{X}$ via the time-varying directed graph $\mathcal{G}_t$. In addition, by the primal--dual protocol \eqref{Algorithm1-eq2}--\eqref{Algorithm1-eq4}, agent $i$ updates its local primal variable ${x_{i,t}} \in \mathbb{X}$ and dual variable ${q_{i,t}} \in \mathbb{R}_ + ^{{m_i}}$, where ${\omega _{i,t}}$ is the updating direction of the local primal variable, ${\alpha _t}$ and ${\beta _t}$ are the updating step-sizes of the local primal and dual variables, respectively, and ${\gamma _t}$ is the regularization parameter. The current decision of agent~$i$ is broadcasted only if norm of the difference between the decision and the last broadcasted decision is not less than the current event-triggering threshold ${{\tau _t}}$.

The following assumption is made for the event-triggering threshold.
\vspace{-3mm}
\begin{assumption}
The event-triggering threshold sequence $\{ {\tau _t} \ge 0 \}$ satisfies ${\tau _{t + 1}} \le {\tau _t}$ for all $t \ge 2$.
\end{assumption}


\subsection{Performance metrics}
We adopt network regret and cumulative constraint violation as performance metrics to evaluate Algorithm~1 as in \cite{Yi2023}, which are respectively defined as
\begin{flalign}
{\rm{Net}\mbox{-}\rm{Reg}}( {\{ {{x_{i,t}}} \},{y_{[ T ]}}} ) &:= \frac{1}{n}\sum\limits_{i = 1}^n {\sum\limits_{t = 1}^T {{f_t}( {{x_{i,t}}} )} }  - \sum\limits_{t = 1}^T {{f_t}( {{y_t}} )}, \label{regret-eq1}
\end{flalign}
\begin{flalign}
{\rm{Net}\mbox{-}\rm{CCV}}( {\{ {{x_{i,t}}} \}} ) &:= \frac{1}{n}\sum\limits_{i = 1}^n {\sum\limits_{t = 1}^T {\| {{{[ {{g_t}( {{x_{i,t}}} )} ]}_ + }} \|} }, \label{CCV-eq2}
\end{flalign}
where ${y_{[ T ]}} = ( {{y_1}, \cdot  \cdot  \cdot ,{y_T}} )$ is a benchmark.

Note that the network cumulative constraint violation \eqref{CCV-eq2} is appropriate in some safety-critical applications where constraints should not be checked across rounds  in cumulation.
Network cumulative constraint violation~\eqref{CCV-eq2} avoids the negligence of some constraint violations barbecue of the effect of some strictly feasible decisions at other iterations.
Therefore it is stricter than network constraint violation adopted in \cite{Yi2020, Li2020, Yi2021b} which takes the summation across rounds before the projection operation  $[  \cdot  ]_ + $.

Moreover, we consider dynamic and static network regret, i.e., ${\rm{Net}\mbox{-}\rm{Reg}}( {\{ {{x_{i,t}}} \},{\check{x}_{[ T ]}^ *}} )$ and ${\rm{Net}\mbox{-}\rm{Reg}}( {\{ {{x_{i,t}}} \},{\hat x_{[ T ]}^*}} )$. For dynamic network regret, the benchmark $\check{x}_{[ T ]}^ *  = ( {\check{x}_1^ * , \cdots,\check{x}_T^ * } )$ is the optimal decision sequence, where $\check{x}_t^ * \in \mathbb{X}$ is the minimizer of ${f_t}( x )$ subject to ${c_t}( x ) \le {\mathbf{0}_m}$.
For static network regret, the benchmark $\hat x_{[ T ]}^ *  = ( {\hat{x}^ * , \cdots,\hat{x}^ * } )$ is the optimal static decision sequence, where $\hat{x}^ * \in \mathbb{X}$ is the minimizer of $\sum\nolimits_{t = 1}^T {{f_t}( x )} $ subject to ${g_t}( x ) \le {\mathbf{0}_m}$ for $t \in [T]$.

\section{Performance analysis}
In this section, we establish dynamic network regret and network cumulative constraint violation bounds for Algorithm~1.

Firstly, inspired by \cite{Cao2021, Xiong2023, Yi2023}, we specially design the updating step-size sequences of the local primal and dual variables, and the regularization parameter sequence of Algorithm~1 in the following theorem.

\begin{theorem}\label{thm1}
Suppose Assumptions 1--3 hold. Let $\{ {{x_{i,t}}} \}$ be the sequences generated by Algorithm~1 with
\begin{flalign}
&{\alpha _t} = \sqrt {\frac{{{\Psi _t}}}{t}}, {\beta _t} = \frac{1}{{{t^\kappa }}}, {\gamma _t} = \frac{1}{{{t^{1 - \kappa }}}}, \forall t \in {\mathbb{N}_ + }, \label{theorem1-eq1}
\end{flalign}
where ${\Psi _t} = \sum\nolimits_{\KP{s} = 1}^t {{\tau _\KP{s}}} $, $\kappa  \in ( {0,1} )$ are constants. Then, for any $T \in {\mathbb{N}_ + }$ and any comparator sequence ${y_{[ T ]}} \in {\mathcal{X}_T}$,
\begin{flalign}
{{\rm{Net}\mbox{-}\rm{Reg}}( {\{ {{x_{i,t}}} \},{y_{[ T ]}}} )} &= \mathcal{O}( {T^\kappa } + \sqrt {{\Psi_T} T}  + \sqrt {{{\Psi ^{ - 1}_T}}T} {P_T} ), \label{theorem1-eq2}\\
{{\rm{Net} \mbox{-} \rm{CCV}}( {\{ {{x_{i,t}}} \}} )}
&= \mathcal{O}( {T^{1 - \kappa /2}} + \sqrt[4]{{{\Psi_T} {T^3}}} ), \label{theorem1-eq3}
\end{flalign}
where ${P_T} = \sum\nolimits_{t = 1}^{T - 1} {\| {{y_{t + 1}} - {y_t}} \|} $ is the path--length of the benchmark ${y_{[ T ]}}$.
\end{theorem}


\begin{remark}\label{rem1}
Theorem~1 establishes the dynamic network regret bound \eqref{theorem1-eq2} and network cumulative constraint violation bound \eqref{theorem1-eq3} for Algorithm~1.
If the path--length of the benchmark grows sublinearly, and $\tau_t$ converges to zero, i.e., $\sum\nolimits_{k = 1}^t {{\tau _k}}$ grows sublinearly,
then these bounds are sublinear.
Moreover, note that $\sqrt {{\Psi_T} T}$ and $\sqrt {{{\Psi^{ - 1}_T}}T}$ are derived for ${\alpha _t}$ given by \eqref{theorem1-eq1}.
Due to the event-triggering threshold ${\tau _t}$, the bound~\eqref{theorem1-eq2} is different from the well-known best regret bound $\mathcal{O}( {\sqrt{T}})$ achieved by the centralized online algorithm in~\cite{Zinkevich2003}.
\end{remark}

\vspace{-3mm}
\begin{remark}\label{rem2}
The step-size ${\beta _t}$ of the dual variable and the regularization parameter ${\gamma _t}$ are specially designed to bound the term $\sum\nolimits_{t = 1}^T {(\frac{1}{{{\gamma _t}}} - \frac{1}{{{\gamma _{t + 1}}}} + {\beta _{t + 1}})}$ of \eqref{lemma4-eq1} and \eqref{lemma4-eq2} in Lemma~4.
\end{remark}

\vspace{-3mm}
\begin{remark}\label{rem3}
When $\kappa  = 1/2$, we have $\sqrt {{\Psi _T}T}  > {T^\kappa }$ and $\sqrt[4]{{{\Psi _T}{T^3}}} > {T^{1 - \kappa /2}}$. Therefore the bounds \eqref{theorem1-eq2} and~\eqref{theorem1-eq3} become $\mathcal{O}( \sqrt {{\Psi_T} T}  + \sqrt {{{\Psi^{ - 1}_T} }T} {P_T} )$ and $\mathcal{O}( \sqrt[4]{{{\Psi_T} {T^3}}} )$, respectively.
Note that it follows from the bounds \eqref{theorem1-eq2} and \eqref{theorem1-eq3} that there exists a trade-off between the dynamic network regret and network cumulative constraint violation, that is, as $\kappa$ increases, the dynamic network regret increases, while the network cumulative constraint violation decreases.
\end{remark}

\vspace{-3mm}
\begin{remark}\label{rem4}
The proof of Theorem~1 has substantial differences compared to the proof of Theorem~1 in \cite{Yi2023}.
More specifically, in our Algorithm~1, the agents broadcast the current local decisions only if the event-triggering condition is satisfied.
Therefore, the resulting decision sequence is different with Algorithm~1 without event-triggered communication in~\cite{Yi2023} although the updating rules are same.
This critical difference leads to challenges in theoretical proof because we need to reanalyse all results related to the local decisions, e.g., the disagreement among agents, the global loss and constraint, dynamic network regret and network cumulative constraint violation bounds.
To tackle this challenge, we derive the upper bounds for the difference between the last broadcasted local decisions and the current local decisions using the current event-triggering threshold. Consequently, the established dynamic network regret bound \eqref{theorem1-eq2} and network cumulative constraint violation bound \eqref{theorem1-eq3} are subject to event-triggering threshold.
\end{remark}

We then select the event-triggering threshold sequence produced by ${\tau _t} = 1/{t^\theta }$ in the following corollary, which is also adopted by the distributed online algorithms in \cite{Cao2021, Paul2022, Oakamoto2023, Xiong2023} since $1/{t^\theta }$ naturally satisfies two conditions: 1) it is non-increasing; 2) it converges to zero.
\begin{corollary}\label{cor1}
Under the same conditions as in Theorem 1 with ${\tau _t} = 1/{t^\theta }$ and $\theta  > 0$, for any $T \in {\mathbb{N}_ + }$, it holds that
\begin{flalign}
\nonumber
&{{\rm{Net}\mbox{-}\rm{Reg}}( {\{ {{x_{i,t}}} \},{y_{[ T ]}}} )} \\
&= \left\{ \begin{array}{l}
\mathcal{O}( {{T^{\max \{ {\kappa ,1 - \theta /2} \}}} + {T^{\theta /2}}{P_T}} ), \;\;\;\;\;\;\;\;\;\;\mathrm{if} \; 0 < \theta < 1, \\
\mathcal{O}\big( {{T^\kappa } + \sqrt {T\log ( T )}  + \sqrt {\frac{T}{{\log ( T )}}} {P_T}} \big), \mathrm{if} \; \theta = 1, \\
\mathcal{O}( {{T^{\max \{ {\kappa ,1/2} \}}} + {T^{1/2}}{P_T}} ), \;\;\;\;\;\;\;\;\;\;\;\;\;\;\mathrm{if} \; \theta > 1,
\end{array} \right. \label{corollary1-eq1} \\
\nonumber
&{{\rm{Net} \mbox{-} \rm{CCV}}( {\{ {{x_{i,t}}} \}} )} \\
&= \left\{ \begin{array}{l}
\mathcal{O}( {{T^{\max \{ {1 - \kappa /2,1 - \theta /4} \}}}} ), \;\;\;\;\;\;\;\;\;\;\;\;\;\;\;\;\;\;\;\;\mathrm{if} \; 0 < \theta < 1, \\
\mathcal{O}( {{T^{1 - \kappa /2}} + \sqrt[4]{{{T^3}\log ( T )}}} ), \;\;\;\;\;\;\;\;\;\;\;\;\;\;\mathrm{if} \; \theta = 1, \\
\mathcal{O}( {{T^{\max \{ {1 - \kappa /2,3/4} \}}}} ), \;\;\;\;\;\;\;\;\;\;\;\;\;\;\;\;\;\;\;\;\;\;\;\;\mathrm{if} \; \theta > 1.
\end{array} \right. \label{corollary1-eq2}
\end{flalign}
\end{corollary}

\begin{remark}\label{rem5}
The bounds \eqref{corollary1-eq1} and \eqref{corollary1-eq2} are sublinear if the path--length ${P_T}$ grows sublinearly. If $\theta  > 1$, the bound \eqref{corollary1-eq1} recovers the results achieved by the centralized online algorithm in \cite{Hall2015} and the distributed event-triggered online algorithm in \cite{Oakamoto2023}.
Moreover, our Algorithm~1 is able to handle (time-varying) inequality constraints, whereas the algorithms in \cite{Oakamoto2023, Hall2015} are limited to a ball set and a box set, respectively.
\end{remark}

Next, we consider the event-triggering threshold sequence produced by ${\tau _t} =1/{c^t}$ in the following corollary, which is also adopted in distributed optimization with event-triggered communication, see, e.g., \cite{Seyboth2013, Yang2016b, Ding2017, Ge2020, Yang2022}.
\begin{corollary}\label{cor2}
Under the same conditions as in Theorem 1 with ${\tau _t} = 1/{c^t}$ and $c  > 1$, for any $T \in {\mathbb{N}_ + }$, it holds that
\begin{flalign}
{{\rm{Net}\mbox{-}\rm{Reg}}( {\{ {{x_{i,t}}} \},{y_{[ T ]}}} )} &= \mathcal{O}( {{T^{\max \{ {\kappa ,1/2} \}}} + {T^{1/2}}{P_T}} ), \label{corollary2-eq1}\\
{{\rm{Net} \mbox{-} \rm{CCV}}( {\{ {{x_{i,t}}} \}} )} &= \mathcal{O}( {{T^{\max \{ {1 - \kappa /2,3/4} \}}}} ). \label{corollary2-eq2}
\end{flalign}
\end{corollary}

\begin{remark}\label{rem6}
The bounds \eqref{corollary2-eq1} and \eqref{corollary2-eq2} recover the results achieved in Corollary 1 with $\theta  > 1$. Moreover, the bound~\eqref{corollary2-eq1} recovers the results achieved by the centralized online algorithm in \cite{Hall2015} and the distributed event-triggered online algorithm in \cite{Oakamoto2023}.
\end{remark}

Note that in (14), the event-triggering threshold also affects the updating step-size $\alpha_t$ of the local primal variable in addition to communication between the agents. To avoid that, we appropriately design new parameter sequences for Algorithm~1 in the following theorem.
\begin{theorem}\label{thm2}
Suppose Assumptions 1--3 hold. Let $\{ {{x_{i,t}}} \}$ be the sequences generated by Algorithm~1 with
\begin{flalign}
&{\alpha _t} = \frac{{{\alpha _0}}}{{{t^{{\theta _1}}}}}, {\beta _t} = \frac{1}{{{t^{{\theta _2}}}}}, {\gamma _t} = \frac{1}{{{t^{1 - {\theta _2}}}}}, {\tau _t} = \frac{{{\tau _0}}}{{{t^{{\theta _3}}}}}, \forall t \in {\mathbb{N}_ + }, \label{theorem2-eq1}
\end{flalign}
where ${\theta _1} \in ( {0,1} )$, ${\theta _2}  \in ( {0,1} )$, ${\alpha _0}$, ${\tau _0}$ and $\theta _3$ are positive constants. Then, for any $T \in {\mathbb{N}_ + }$ and any comparator sequence ${y_{[ T ]}} \in {\mathcal{X}_T}$,
\begin{flalign}
\nonumber
&{{\rm{Net}\mbox{-}\rm{Reg}}( {\{ {{x_{i,t}}} \},{y_{[ T ]}}} )} \\
&= \left\{ \begin{array}{l}
\mathcal{O}\big( {\alpha _0}{T^{1 - {\theta _1}}} + {T^{{\theta _2}}} + \frac{{{\tau _0}}}{{{\alpha _0}}}{T^{1 + {\theta _1} - {\theta _3}}}\\
\;\;\;\;\;\;\;\;\;\;\;\;\;\;\;\;\;\;\; + \frac{{{T^{{\theta _1}}}( {1 + {P_T}} )}}{{{\alpha _0}}} \big), \;\;\;\;\mathrm{if} \; {\theta _1} < {\theta _3} < 1 + {\theta _1},\\
\mathcal{O}\big( {\alpha _0}{T^{1 - {\theta _1}}} + {T^{{\theta _2}}} + \frac{{{\tau _0}}}{{{\alpha _0}}}\log ( T ) \\
\;\;\;\;\;\;\;\;\;\;\;\;\;\;\;\;\;\;\; + \frac{{{T^{{\theta _1}}}( {1 + {P_T}} )}}{{{\alpha _0}}} \big),\;\;\;\;\mathrm{if} \; {\theta _3} = 1 + {\theta _1},\\
\mathcal{O}\big( {\alpha _0}{T^{1 - {\theta _1}}} + {T^{{\theta _2}}} + \frac{{{\tau _0}}}{{{\alpha _0}}} \\
\;\;\;\;\;\;\;\;\;\;\;\;\;\;\;\;\;\;\; + \frac{{{T^{{\theta _1}}}( {1 + {P_T}} )}}{{{\alpha _0}}} \big),\;\;\;\;\mathrm{if} \; {\theta _3} > 1 + {\theta _1},
\end{array} \right. \label{theorem2-eq2} \\
\nonumber
&{{\rm{Net} \mbox{-} \rm{CCV}}( {\{ {{x_{i,t}}} \}} )} \\
&=\left\{ \begin{array}{l}
\mathcal{O}( \sqrt {{\alpha _0}} {T^{1 - {\theta _1}/2}} + {T^{1 - {\theta _2}/2}}\\
\;\;\;\;\;\;\;\;\;\;\;\;\;\;\;\;\;\;\; + \sqrt {{\tau _0}} {T^{1 - {\theta _3}/2}} ), \;\; \mathrm{if} \; {\theta _1} < {\theta _3} < 1,\\
\mathcal{O}\big( \sqrt {{\alpha _0}} {T^{1 - {\theta _1}/2}} + {T^{1 - {\theta _2}/2}}\\
\;\;\;\;\;\;\;\;\;\;\;\;\;\;\;\;\;\;\; + \sqrt {{\tau _0}T\log ( T )} \big),\mathrm{if} \; {\theta _3} = 1,\\
\mathcal{O}( \sqrt {{\alpha _0}} {T^{1 - {\theta _1}/2}} + {T^{1 - {\theta _2}/2}} \\
\;\;\;\;\;\;\;\;\;\;\;\;\;\;\;\;\;\;\; + \sqrt {{\tau _0}} {T^{1/2}} ), \;\;\;\;\;\;\; \mathrm{if} \; {\theta _3} > 1.
\end{array} \right. \label{theorem2-eq3}
\end{flalign}
\end{theorem}


\begin{remark}\label{rem7}
Theorem~2 establishes dynamic network regret bound \eqref{theorem2-eq2} and network cumulative constraint violation bound \eqref{theorem2-eq3} for Algorithm~1.
If the path--length of the benchmark grows sublinearly, then these bounds are sublinear.
Moreover, the bounds \eqref{theorem2-eq2} and \eqref{theorem2-eq3} show the impact of event-triggering threshold on dynamic network regret and network cumulative constraint violation, that is, the larger ${\tau _0}$ is, the larger the bounds are.
When ${\tau _0} = 0$, i.e., without event-triggered communication, the bounds are the same as the results achieved by distributed online algorithms in \cite{Yi2023} when choosing ${\theta _1} = {\theta _2}$.
\end{remark}

\vspace{-3mm}
\begin{remark}\label{rem8}
Note that to achieve dynamic network regret and network cumulative constraint violation, ${\tau _t}$ cannot be selected as a fixed positive constant in Theorems 1 and 2 due to the conditions that event-triggering threshold $\tau_t$ converges to zero in Theorem 1 and ${\theta _3} > {\theta _1}$ in Theorem 2, respectively.
\end{remark}

\vspace{-3mm}
\begin{remark}\label{rem9}
It should be pointed out that the smallest network cumulative constraint violation bounds in Corollaries~1 and~2 are $\mathcal{O}( {{T^{3/4}}} )$, which are reduced to $\mathcal{O}( {{T^{1/2}}} )$ in Theorem~2. Moreover, the smallest dynamic network regret bounds in Corollaries~1 and~2 are $\mathcal{O}( {{T^{1/2}}} )$, which are recovered in Theorem~2. In addition, the bounds \eqref{theorem2-eq2} and \eqref{theorem2-eq3} recover the results achieved by the centralized online algorithms in \cite{Hall2015}, the distributed online algorithms in \cite{Yi2023}, and the distributed event-triggered online algorithm in \cite{Oakamoto2023}.
\end{remark}

\vspace{-3mm}
\begin{remark}\label{rem10}
By replacing the any benchmark ${y_{[ T ]}}$ with the optimal static decision sequence $\hat x_{[ T ]}^*$, we have ${P_T} \equiv 0$, and then the static network regret and cumulative constraint violation bounds for Algorithm~1 with corresponding parameter sequences can be easily established based on the results in Theorems~1 and~2, and Corollaries~1 and~2, respectively, which are the same as \eqref{theorem1-eq2}--\eqref{corollary2-eq2}, \eqref{theorem2-eq2}, and \eqref{theorem2-eq3} with ${P_T} \equiv 0$, respectively. The bounds recover the results achieved by the centralized online algorithm in \cite{Zinkevich2003}, the distributed online algorithms in \cite{Yuan2022, Yi2023}, and the distributed event-triggered online algorithms in \cite{Cao2021, Paul2022, Xiong2023}.
\end{remark}

\vspace{-4mm}
\section{Numerical example}
Consider a distributed online linear regression problem with time-varying linear inequality constraints over a network of $n$ agents in \cite{Yi2023}. At each iteration $t$, agent $i$ for $i \in [ n ]$ accesses to the local loss and constraint functions, i.e., ${f_{i,t}}( x ) = \frac{1}{2}\| {{A_{i,t}}x - {\vartheta _{i,t}}\|^2}$ and ${g_{i,t}}( x ) = {B_{i,t}}x - {b_{i,t}}$, where each entry of ${A_{i,t}} \in {\mathbb{R}^{{q_i} \times p}}$ is randomly generated from the uniform distribution in the interval $[ { - 1,1} ]$, ${\vartheta _{i,t}} = {A_{i,t}}{\mathbf{1}_p} + {\zeta _{i,t}}$, where ${\zeta _{i,t}}$ is a standard normal random vector, and each entry of ${B_{i,t}} \in {\mathbb{R}^{{m_i} \times p}}$ and ${b_{i,t}} \in {\mathbb{R}^{{m_i}}}$ is randomly generated from the uniform distribution in the interval $[ {0,2} ]$ and $[ {0,1} ]$, respectively.
We set $n = 100$, ${q_i} = 4$, $p = 10$, ${m_i} = 2$, $\mathbb{X} = {[ { - 5,5} ]^p}$.
At each iteration $t$, we use an undirected random graph to model the underlying communication topology.
Specifically, connections between the agents are random and the probability of two agents being connected is $0.1$. To make sure that Assumption~2 is satisfied, we add edges~$( {i,i + 1} )$ for $i \in [ {n - 1} ]$, and let
${[ {{W_t}} ]_{ij}} = \frac{1}{n}$ if $( {j,i} ) \in {\mathcal{E}_t}$ and ${[ {{W_t}} ]_{ii}} = 1 - \sum\nolimits_{j = 1}^n {{{[ {{W_t}} ]}_{ij}}} $.

\begin{figure}[!ht]
 \centering
  \includegraphics[width=6.5cm]{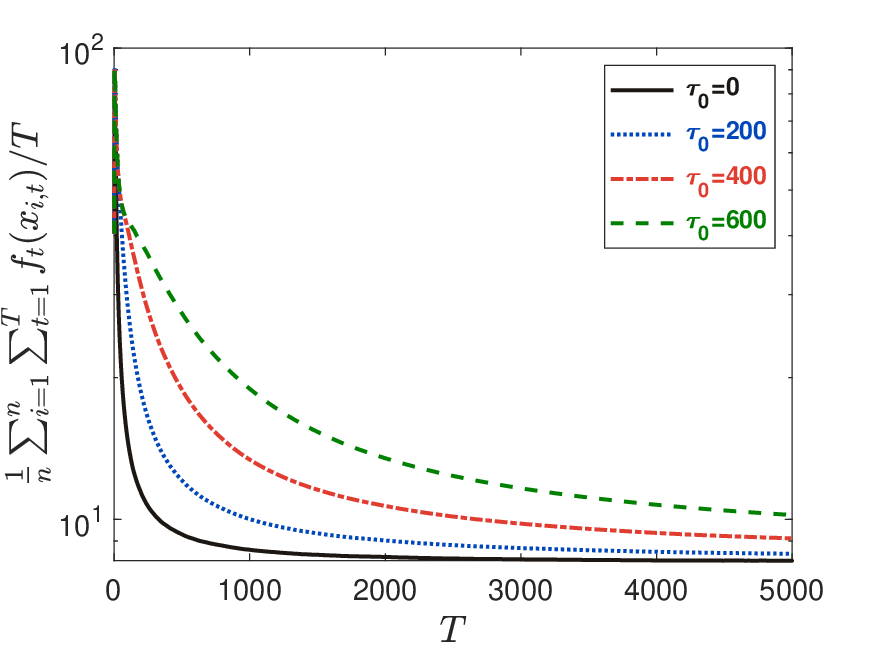}
  \caption{Evolutions of $\frac{1}{n}\sum\nolimits_{i = 1}^n {\sum\nolimits_{t = 1}^T {{f_t}( {{x_{i,t}}} )} } /T$.}
\end{figure}

\begin{figure}[!ht]
 \centering
  \includegraphics[width=6.5cm]{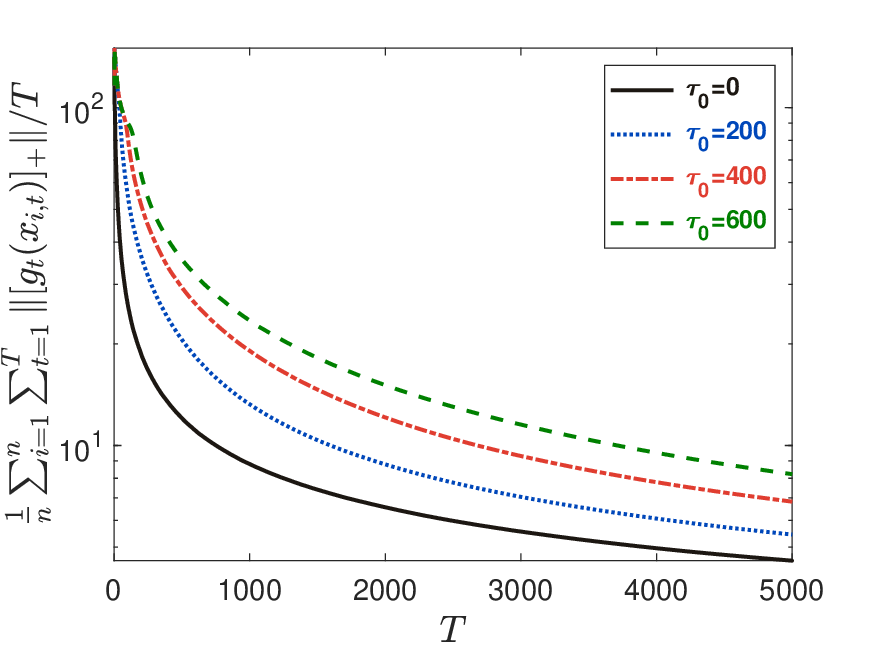}
  \caption{Evolutions of $\frac{1}{n}\sum\nolimits_{i = 1}^n {\sum\nolimits_{t = 1}^T {\| {{{[ {{g_t}( {{x_{i,t}}} )} ]}_ + }} \|} } /T$.}
\end{figure}

\begin{figure}[!ht]
 \centering
  \includegraphics[width=6.5cm]{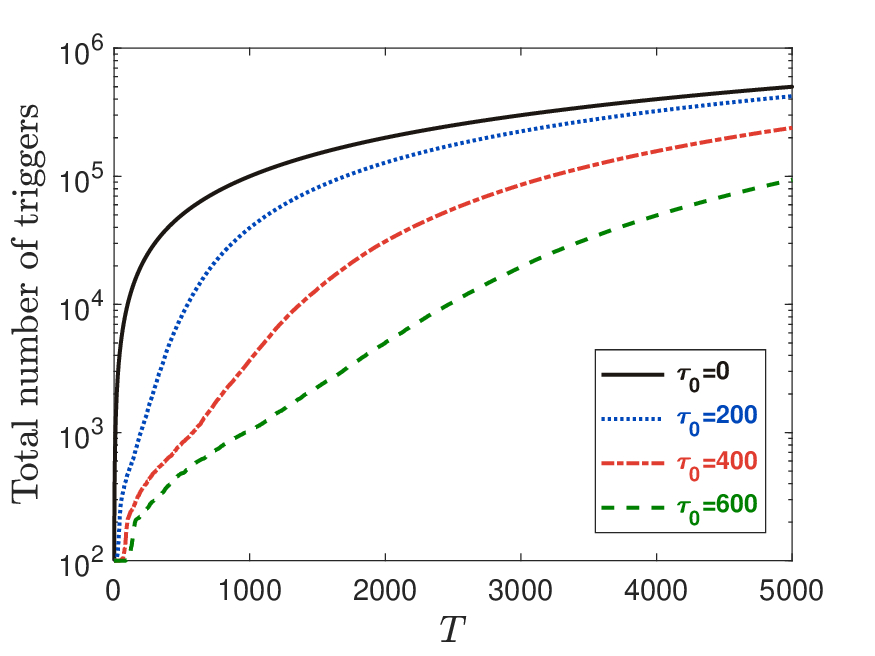}
  \caption{Evolutions of total number of triggers.}
\end{figure}
Set ${\alpha _t} = 1/{t^{1/2}}$, ${\beta _t} = 1/{t^{1/2}}$, ${\gamma _t} = 1/{t^{1/2}}$ and ${\tau _t} = {\tau _0}/t$ for Algorithm~1. To explore the impact of different event-triggering threshold sequences on network regret and cumulative constraint violation, we select ${\tau _0} = 0$, ${\tau _0} = 200$, ${\tau _0} = 400$, and ${\tau _0} = 600$, respectively.
With different values of ${\tau _0}$, Figs. 1--3 illustrate the evolutions of the average cumulative loss $\frac{1}{n}\sum\nolimits_{i = 1}^n {\sum\nolimits_{t = 1}^T {{f_t}( {{x_{i,t}}} )} } /T$, the average cumulative constraint violation $\frac{1}{n}\sum\nolimits_{i = 1}^n {\sum\nolimits_{t = 1}^T {\| {{{[ {{g_t}( {{x_{i,t}}} )} ]}_ + }} \|} } /T$, and total number of triggers, respectively.
The results show that as ${\tau _0}$ increases, the average cumulative loss and the average cumulative constraint violation increase, and the total number of triggers decreases, which are consistent with the theoretical results in Theorem~2.

\vspace{-4mm}
\section{Conclusions}
This paper considered the distributed online convex optimization problem with time-varying inequality constraints. We proposed the distributed event-triggered online primal--dual algorithm to reduce communication overhead for a time-varying directed graph.
We analyzed the network regret and cumulative constraint violation for the proposed algorithm.
Our theoretical results are comparable to the results achieved by the related centralized and distributed online algorithms and distributed event-triggered online algorithms when the event-triggering threshold is properly chosen.
In the future, we will investigate the distributed event-triggered online primal--dual algorithm with bandit feedback as gradient information is unavailable in many real-world applications. Moreover, we will consider to eliminate the need of doubly stochastic in the future.

\vspace{-4mm}

\section*{Acknowledgement}
This work was supported by the National Key Research and Development Program of China under Grant 2022YFB3305904, the National Natural
Science Foundation of China under Grants 62133003 \& 61991403.



\appendix

\section{Preliminary Lemmas}
To prove Theorems~1 and~2, some preliminary results are derived in this section.

We follow the proofs in \cite{Yi2023}, but must take care to consider the event-triggering threshold ${\tau _t}$, which affects the decision sequence induced by Algorithm~1 such that we need to reanalyse all results related to the local decisions, e.g., the disagreement among agents, the global loss and constraint, dynamic network regret and network cumulative constraint violation bounds.

Firstly, we quantify the disagreement among agents.
\begin{lemma}\label{lem1}
If Assumption 2 holds. For all $i \in [ n ]$ and $t \in {\mathbb{N}_ + }$, ${{\hat x}_{i,t}}$ generated by Algorithm 1 satisfy
\begin{flalign}
\nonumber
\| {{{\hat x}_{i,t}} - {{\bar x}_t}} \| &\le \tau {\lambda ^{t - 2}}\sum\limits_{j = 1}^n {\| {{{\hat x}_{j,1}}} \|}  + \frac{1}{n}\sum\limits_{j = 1}^n {\| {\hat{\varepsilon} _{j,t - 1}^x} \|}  + \| {\hat{\varepsilon} _{i,t - 1}^x} \| \\
& \;\;+ \tau \sum\limits_{s = 1}^{t - 2} {{\lambda ^{t - s - 2}}} \sum\limits_{j = 1}^n {\| {\hat{\varepsilon} _{j,s}^x} \|}, \label{lemma1-eq1}
\end{flalign}
where ${{\bar x}_t} = \frac{1}{n}\sum\nolimits_{j = 1}^n {{{\hat x}_{j,t}}} $ and $\hat{\varepsilon} _{i,t - 1}^x = {{\hat x}_{i,t}} - {z_{i,t}}$.
\end{lemma}
\begin{proof}
The proof is presented in Lemma 4 of \cite{Yi2023}.
\end{proof}

We then give a result on the evolution of local dual variables.
\begin{lemma}\label{lem2}
Suppose Assumptions 1 and 2 hold, and ${\gamma _t}{\beta _t} \le 1$, $t \in {\mathbb{N}_ + }$. For all $i \in [ n ]$ and $t \in {\mathbb{N}_ + }$, the sequences ${q_{i,t}}$ generated by Algorithm 1 satisfy
\begin{flalign}
\nonumber
{\Delta _{i,t}}( {{\mu _i}} ) &\le {\varpi _1}{\gamma _t} + q_{i,t - 1}^T{b_{i,t}} - \mu _i^T{[ {{g_{i,t - 1}}( {{x_{i,t - 1}}} )} ]_ + }  \\
 & \;\;+ \frac{1}{2}{\beta _t}{\| {{\mu _i}} \|^2} + {F_2}\| {{\mu _i}} \|\| {{x_{i,t}} - {x_{i,t - 1}}} \|, \label{lemma2-eq1}
\end{flalign}
where ${\Delta _{i,t}}( {{\mu _i}} ) = \frac{1}{{2{\gamma _t}}}\big( {{{\| {{q_{i,t}} - {\mu _i}} \|}^2} - ( {1 - {\beta _t}{\gamma _t}} ){{\| {{q_{i,t - 1}} - {\mu _i}} \|}^2}} \big)$ with ${{\mu _i}}$ being an arbitrary vector in $\mathbb{R}_ + ^{{m_i}}$, ${\varpi _1} = 2{( {{F_1} + {F_2}R( \mathbb{X} )} )^2}$, and ${b_{i,t}} = {[ {{g_{i,t - 1}}( {{x_{i,t - 1}}} )} ]_ + } + {\big( {\partial {{[ {{g_{i,t - 1}}( {{x_{i,t - 1}}} )} ]}_ + }} \big)^T}( {{x_{i,t}} - {x_{i,t - 1}}} )$.
\end{lemma}
\begin{proof}
The proof is presented in Lemma 5 of \cite{Yi2023}.
\end{proof}

Next, we present network regret bound at one slot.
\begin{lemma}\label{lem3}
Suppose Assumptions 1--3 hold. For all $i \in [ n ]$, let $\{ {{x_{i,t}}}\}$ be the sequences generated by Algorithm 1 and $\{ {{y_t}} \}$ be an arbitrary sequence in $\mathbb{X}$, then
\begin{flalign}
\nonumber
& \;\;\;\;\;\frac{1}{n}\sum\limits_{i = 1}^n {{f_t}( {{x_{i,t}}} )}  - {f_t}( {{y_t}} ) \\
\nonumber
& \le \frac{1}{n}\sum\limits_{i = 1}^n {q_{i,t}^T\big( {{{[ {{g_{i,t}}( {{y_t}} )} ]}_ + } - {b_{i,t + 1}}} \big) - \frac{1}{{2n{\alpha _{t + 1}}}}} \sum\limits_{i = 1}^n {{{\| \varepsilon _{i,t}^x \|}^2}} \\
\nonumber
& \;\;+ \frac{1}{n}\sum\limits_{i = 1}^n {{F_2}( {2\| {{{\hat x}_{i,t}} - {{\bar x}_t}} \| + \| {{x_{i,t}} - {x_{i,t + 1}}} \|} )}  + 2{F_2}{\tau _t} \\
\nonumber
& \;\;+ \frac{1}{{2n{\alpha _{t + 1}}}}\sum\limits_{i = 1}^n ( {{\| {{y_t} - {z_{i,t + 1}}} \|}^2} - {{\| {{y_{t + 1}} - {z_{i,t + 2}}} \|}^2} \\
& \;\;+ {{\| {{y_{t + 1}} - {{\hat x}_{i,t + 1}}} \|}^2} - {{\| {{y_t} - {x_{i,t + 1}}} \|}^2} ), \label{lemma3-eq1}
\end{flalign}
where $\varepsilon _{i,t}^x = {x_{i,t + 1}} - {z_{i,t + 1}}$.
\end{lemma}
\begin{proof}
From Assumption 1, for $i \in [ n ]$, $t \in {\mathbb{N}_ + }$, $x,y \in \mathbb{X}$, we have
\begin{subequations}
\begin{flalign}
| {{f_{i,t}}( x ) - {f_{i,t}}( y )} | &\le {F_2}\| {x - y} \|, \label{lemma3-proof-eq1a}\\
\| {{g_{i,t}}( x ) - {g_{i,t}}( y )} \| &\le {F_2}\| {x - y} \|. \label{lemma3-proof-eq1b}
\end{flalign}
\end{subequations}

We make a critical analysis to obtain a bound for the difference between decisions induced by Algorithm~1 and decisions triggered at each iteration in the following.

From Algorithm~1, for any $t \in {\mathbb{N}_ + }$, if $\| {{x_{i,t + 1}} - {{\hat x}_{i,t}}} \| \ge {\tau _{t + 1}}$, then $\| {{{\hat x}_{i,t + 1}} - {x_{i,t + 1}}} \| \le {\tau _{t + 1}}$. If $\| {{x_{i,t + 1}} - {{\hat x}_{i,t}}} \| < {\tau _{t + 1}}$, then ${{\hat x}_{i,t + 1}} = {{\hat x}_{i,t}}$ and we still have $\| {{{\hat x}_{i,t + 1}} - {x_{i,t + 1}}} \| \le {\tau _{t + 1}}$. Therefore, we always have $\| {{{\hat x}_{i,t}} - {x_{i,t}}} \| \le {\tau _t}$ for any $t \ge 2$, $i \in [ n ]$. Recall that ${{\hat x}_{i,1}} = {x_{i,1}}$. Thus, $\| {{{\hat x}_{i,t}} - {x_{i,t}}} \| \le {\tau _t}$ for any $t \ge 1$, $i \in [ n ]$.

From \eqref{lemma3-proof-eq1a} and $\| {{{\hat x}_{i,t}} - {x_{i,t}}} \| \le {\tau _t}$, we have
\begin{flalign}
\nonumber
&\;\;\;\;\;\frac{1}{n}\sum\limits_{i = 1}^n {{f_t}( {{x_{i,t}}} )} \\
\nonumber
& = \frac{1}{n}\sum\limits_{i = 1}^n {{f_{i,t}}( {{x_{i,t}}} )}  + \frac{1}{{{n^2}}}\sum\limits_{i = 1}^n {\sum\limits_{j = 1}^n {\big( {{f_{j,t}}( {{x_{i,t}}} ) - {f_{j,t}}( {{x_{j,t}}} )} \big)} } \\
\nonumber
& \le \frac{1}{n}\sum\limits_{i = 1}^n {{f_{i,t}}( {{x_{i,t}}} )}  + \frac{{{F_2}}}{{{n^2}}}\sum\limits_{i = 1}^n {\sum\limits_{j = 1}^n {\| {{x_{i,t}} - {x_{j,t}}} \|} } \\
& \le \frac{1}{n}\sum\limits_{i = 1}^n {{f_{i,t}}( {{x_{i,t}}} )}  + \frac{{2{F_2}}}{n}\sum\limits_{i = 1}^n {\| {{{\hat x}_{i,t}} - {{\bar x}_t}} \|}  + 2{F_2}{\tau _t}. \label{lemma3-proof-eq2}
\end{flalign}

It then follows from the proof of Lemma 6 of \cite{Yi2023} that \eqref{lemma3-eq1} holds.
\end{proof}
\begin{lemma}\label{lem4}
Suppose Assumptions 1--3 hold, and ${\gamma _t}{\beta _t} \le 1$, $t \in {\mathbb{N}_ + }$. For all $i \in [ n ]$, let $\{ {{x_{i,t}}}\}$ be the sequences generated by Algorithm~1. Then, for any comparator sequence ${y_{[ T ]}} \in~{\mathcal{X}_T}$,
\begin{flalign}
\nonumber
& \;\;\;\;\;{{\rm{Net}\mbox{-}\rm{Reg}}( {\{ {{x_{i,t}}} \},{y_{[ T ]}}} )} \\
\nonumber
& \le 4{F_2}{\varpi _2} + {\varpi _1}\sum\limits_{t = 1}^T {{\gamma _t}}  + 10{\varpi _6}\sum\limits_{t = 1}^T {\alpha _t} + {\varpi _7}\sum\limits_{t = 1}^T {{\tau _t}} \\
\nonumber
&  \;\;+ 2R( \mathbb{X} )\sum\limits_{t = 1}^T {\frac{{{\tau _{t+1}}}}{{{\alpha _{t+1}}}}}  + \frac{{2R{{( \mathbb{X} )}^2}}}{{{\alpha _{T + 1}}}} + \frac{{2R( \mathbb{X} )}}{{{\alpha _T}}}{P_T} 
\end{flalign}
\begin{flalign}
&  \;\;- \frac{1}{{2n}}\sum\limits_{t = 1}^T {\sum\limits_{i = 1}^n {( {\frac{1}{{{\gamma _t}}} - \frac{1}{{{\gamma _{t + 1}}}} + {\beta _{t + 1}}} )} } {\| {{q_{i,t}}} \|^2}, \label{lemma4-eq1} \\
\nonumber
& \;\;\;\;\;\frac{1}{n}\sum\limits_{i = 1}^n {{\| {\sum\limits_{t = 1}^T {{{[ {{g_t}( {{x_{i,t}}} )} ]}_ + }} } \|^2}} \\
\nonumber
& \le 8n{F_1}{F_2}{\varpi _2}T + 2n{F_1}{\varpi _{10}}T\sum\limits_t^T {{\tau _t}} \\
\nonumber
&  \;\;+ 2\big( {\frac{1}{{{\gamma _1}}} + \sum\limits_{t = 1}^T {( {{\beta _t} + {\varpi _9}{\alpha _t}} )} } \big)\big( n{F_1}T + 4n{F_2}{\varpi _2} \\
\nonumber
&  \;\;+ n{\varpi _1}\sum\limits_{t = 1}^T {{\gamma _t}}  + 20n{\varpi _6}\sum\limits_{t = 1}^T {{\alpha _t}}  + n{\varpi _{10}}\sum\limits_{t = 1}^T {{\tau _t}} \\
\nonumber
& \;\;+ 2nR( \mathbb{X} )\sum\limits_{t = 1}^T {\frac{{{\tau _{t + 1}}}}{{{\alpha _{t + 1}}}}}  + \frac{{2nR{{( \mathbb{X} )}^2}}}{{{\alpha _{T + 1}}}} \\
& \;\;- \frac{1}{2}\sum\limits_{t = 1}^T {\sum\limits_{i = 1}^n {( {\frac{1}{{{\gamma _t}}} - \frac{1}{{{\gamma _{t + 1}}}} + {\beta _{t + 1}}} ){\| {{q_{i,t}} - \mu _{ij}^0} \|^2}} }  \big), \label{lemma4-eq2}
\end{flalign}
where ${\varpi _2} = \frac{\tau }{{\lambda ( {1 - \lambda } )}}\sum\limits_{i = 1}^n {\| {{{\hat x}_{i,1}}} \|} $, ${\varpi _3} = 2{F_2} + \frac{{{n^2}{\tau ^2}{F_2}}}{{2{{\left( {1 - \lambda } \right)}^2}}}$, ${\varpi _4} = 2 + \frac{{n\tau }}{{1 - \lambda }}$, ${\varpi _5} = 2{F_2}{\varpi _3} + \frac{{F_2^2}}{4}$, ${\varpi _6} = 2{F_2}{\varpi _3} + {\varpi _5}$, ${\varpi _7} = 4{F_2}{\varpi _4} + 3{F_2}$, ${\varpi _8} = 1 + {\varpi _4}$, ${\varpi _9} = 20{\varpi _6}$, ${\varpi _{10}} = 3{F_2}{\varpi _4} + {F_2}{\varpi _8} + 2F_2$, and $\mu _{ij}^0 = \frac{{\sum\nolimits_{t = 1}^T {{{[ {{g_{i,t}}( {{x_{j,t}}} )} ]}_ + }} }}{{\frac{1}{{{\gamma _1}}} + \sum\nolimits_{t = 1}^T {( {{\beta _t} + {\varpi _9}{\alpha _t}} )} }}$.
\end{lemma}
\begin{proof}
($\mathbf{i}$)
We first provide a loose bound for network regret.

From \eqref{ass-eq1}, we have
\begin{flalign}
\nonumber
& \;\;\;\;\;{\| {{y_{t + 1}} - {{\hat x}_{i,t + 1}}} \|^2} - {\| {{y_t} - {x_{i,t + 1}}} \|^2} \\
\nonumber
& \le \| {{y_{t + 1}} - {y_t} - {{\hat x}_{i,t + 1}} + {x_{i,t + 1}}} \|\| {{y_{t + 1}} - {x_{i,t + 1}} + {y_t} - {x_{i,t + 1}}} \| \\
& \le 4R( \mathbb{X} )\| {{y_{t + 1}} - {y_t}} \| + 4R( \mathbb{X} ){\tau _{t+1}}. \label{lemma4-proof-eq1}
\end{flalign}
From \eqref{lemma2-eq1}, \eqref{lemma3-eq1} and \eqref{lemma4-proof-eq1}, we have
\begin{flalign}
\nonumber
& \;\;\;\;\;\frac{1}{n}\sum\limits_{i = 1}^n {\big( {{\Delta _{i,t + 1}}( {{\mu _i}} ) + \mu _i^T{{[ {{g_{i,t}}( {{x_{i,t}}} )} ]}_ + } - \frac{1}{2}{\beta _{t + 1}}{{\| {{\mu _i}} \|}^2}} \big)} \\
\nonumber
&  \;\;+ \frac{1}{n}\sum\limits_{i = 1}^n {{f_t}( {{x_{i,t}}} )}  - {f_t}( {{y_t}} ) \\
\nonumber
& \le {\varpi _1}{\gamma _{t + 1}} + \frac{1}{n}\sum\limits_{i = 1}^n {{{\tilde \Delta }_{i,t + 1}}( {{\mu _i}} )} \\
\nonumber
& \;\;+ \frac{1}{{2n{\alpha _{t + 1}}}}\sum\limits_{i = 1}^n {( {{{\| {{y_t} - {z_{i,t + 1}}} \|}^2} - {{\| {{y_{t + 1}} - {z_{i,t + 2}}} \|}^2}} )} \\
&  \;\;+ 2{F_2}{\tau _t} + \frac{{2R( \mathbb{X} ){\tau _{t+1}}}}{{{\alpha _{t + 1}}}} + \frac{{2R( \mathbb{X} )}}{{{\alpha _{t + 1}}}}\| {{y_{t + 1}} - {y_t}} \|, \label{lemma4-proof-eq2}
\end{flalign}
where
\begin{flalign}
\nonumber
{{\tilde \Delta }_{i,t + 1}}( {{\mu _i}} ) &= {F_2}( {\| {{\mu _i}} \| + 1} )\| {{x_{i,t}} - {x_{i,t + 1}}} \| \\
\nonumber
&  \;\;+ 2{F_2}\| {{{\hat x}_{i,t}} - {{\bar x}_t}} \| - \frac{1}{{2{\alpha _{t + 1}}}}{\| {\varepsilon _{i,t}^x} \|^2}.
\end{flalign}
From \eqref{ass-eq1} and $\{ {{\alpha _t}} \}$ is non-increasing, we have
\begin{flalign}
\nonumber
& \;\;\;\;\;\sum\limits_{t = 1}^T {\frac{1}{{{\alpha _{t + 1}}}}} ( {{{\| {{y_t} - {z_{i,t + 1}}} \|}^2} - {{\| {{y_{t + 1}} - {z_{i,t + 2}}} \|}^2}} ) 
\end{flalign}
\begin{flalign}
\nonumber
& = \sum\limits_{t = 1}^T \big( {\frac{1}{{{\alpha _t}}}{{\| {{y_t} - {z_{i,t + 1}}} \|}^2} - \frac{1}{{{\alpha _{t + 1}}}}{{\| {{y_{t + 1}} - {z_{i,t + 2}}} \|}^2}} \\
\nonumber
& \;\;+ ( {\frac{1}{{{\alpha _{t + 1}}}} - \frac{1}{{{\alpha _t}}}} ){{\| {{y_t} - {z_{i,t + 1}}} \|}^2} \big) \\
\nonumber
& \le \frac{1}{{{\alpha _1}}}{\| {{y_1} - {z_{i,2}}} \|^2} - \frac{1}{{{\alpha _{T + 1}}}}{\| {{y_{T + 1}} - {z_{i,T + 2}}} \|^2} \\
& \;\;+ \sum\limits_{t = 1}^T {( {\frac{1}{{{\alpha _{t + 1}}}} - \frac{1}{{{\alpha _t}}}} )4R{( \mathbb{X} )^2}}
\le \frac{{4R{( \mathbb{X} )^2}}}{{{\alpha _{T + 1}}}}. \label{lemma4-proof-eq3}
\end{flalign}
From \eqref{lemma4-proof-eq3} and $\{ {{\alpha _t}} \}$ is non-increasing, setting ${y_{T + 1}} = {y_T}$ and ${\mu _i} = {\mathbf{0}_{{m_i}}}$, summing \eqref{lemma4-proof-eq2} over $t \in [ T ]$ gives
\begin{flalign}
\nonumber
& \;\;\;\;\;{{\rm{Net}\mbox{-}\rm{Reg}}( {\{ {{x_{i,t}}} \},{y_{[ T ]}}} )}  + \frac{1}{n}\sum\limits_{t = 1}^T {\sum\limits_{i = 1}^n {{\Delta _{i,t + 1}}( {{\mathbf{0}_{{m_i}}}} )} } \\
\nonumber
& \le {\varpi _1}\sum\limits_{t = 1}^T {{\gamma _{t + 1}}} + \frac{1}{n}\sum\limits_{t = 1}^T {\sum\limits_{i = 1}^n {{{\tilde \Delta }_{i,t + 1}}( {{\mathbf{0}_{{m_i}}}} )} } + \sum\limits_{t = 1}^T {2{F_2}{\tau _t}} \\
& \;\; + \sum\limits_{t = 1}^T {\frac{{2R( \mathbb{X} ){\tau _{t+1}}}}{{{\alpha _{t + 1}}}}}  + \frac{{2R{{( \mathbb{X} )}^2}}}{{{\alpha _{T+1}}}} + \frac{{2R( \mathbb{X} )}}{{{\alpha _T}}}{P_T}. \label{lemma4-proof-eq4}
\end{flalign}

We then establish a lower bound for network regret.

For any $T \in {\mathbb{N}_ + }$, we have
\begin{flalign}
\nonumber
 \sum\limits_{t = 1}^T {{\Delta _{i,t + 1}}( {{\mu _i}} )}
& = \frac{1}{2}\sum\limits_{t = 1}^T {( {\frac{1}{{{\gamma _t}}} - \frac{1}{{{\gamma _{t + 1}}}} + {\beta _{t + 1}}} ){{\| {{q_{i,t}} - {\mu _i}} \|}^2}} \\
& \;\;+ \frac{{{{\| {{q_{i,T + 1}} - {\mu _i}} \|}^2}}}{{2{\gamma _{T + 1}}}} - \frac{{{{\| {{\mu _i}} \|}^2}}}{{2{\gamma _1}}}. \label{lemma4-proof-eq5}
\end{flalign}

Substituting ${\mu _i} = {\mathbf{0}_{{m_i}}}$ into \eqref{lemma4-proof-eq5} yields
\begin{flalign}
\sum\limits_{t = 1}^T {{\Delta _{i,t + 1}}( {{\mathbf{0}_{{m_i}}}} )}
\ge \frac{1}{2}\sum\limits_{t = 1}^T {( {\frac{1}{{{\gamma _t}}} - \frac{1}{{{\gamma _{t + 1}}}} + {\beta _{t + 1}}} ){{\| {{q_{i,t}}} \|}^2}}. \label{lemma4-proof-eq6}
\end{flalign}

We have
\begin{flalign}
\nonumber
\sum\limits_{t = 1}^T {\sum\limits_{s = 1}^{t - 2} {{\lambda ^{t - s - 2}}\sum\limits_{j = 1}^n {\| {\hat{\varepsilon} _{j,s}^x} \|} } }  & = \sum\limits_{t = 1}^{T - 2} {\sum\limits_{j = 1}^n {\| {\hat{\varepsilon} _{j,t}^x} \|\sum\limits_{s = 0}^{T - t - 2} {{\lambda ^s}} } } \\
& \le \frac{1}{{1 - \lambda }}\sum\limits_{t = 1}^{T - 2} {\sum\limits_{j = 1}^n {\| {\hat{\varepsilon} _{j,t}^x} \|} }. \label{lemma4-proof-eq7}
\end{flalign}
and
\begin{flalign}
\nonumber
\| {\hat \varepsilon _{i,t - 1}^x} \| &= \| {{{\hat x}_{i,t}} - {x_{i,t}} + {x_{i,t}} - {z_{i,t}}} \| \\
&\le \| {\varepsilon _{i,t - 1}^x} \| + {\tau _t}. \label{lemma4-proof-eq8}
\end{flalign}

From \eqref{lemma1-eq1}, \eqref{lemma4-proof-eq7} and \eqref{lemma4-proof-eq8}, we have
\begin{flalign}
\nonumber
& \;\;\;\;\;\sum\limits_{t = 1}^T {\| {{{\hat x}_{i,t}} - {{\bar x}_t}} \|}\\
\nonumber
& \le {\varpi _2} + \frac{1}{n}\sum\limits_{t = 1}^T {\sum\limits_{j = 1}^n {\| {\varepsilon _{j,t - 1}^x} \|} }  + \sum\limits_{t = 1}^T {\| {\varepsilon _{i,t - 1}^x} \|} \\
& \;\;+ \frac{\tau }{{1 - \lambda }}\sum\limits_{t = 1}^{T - 2} {\sum\limits_{j = 1}^n {\| {\varepsilon _{j,t}^x} \|} }  + 2\sum\limits_{t = 1}^T {{\tau _t}}  + \frac{{n\tau }}{{1 - \lambda }}\sum\limits_{t = 1}^{T - 2} {{\tau _{t + 1}}}. \label{lemma4-proof-eq9}
\end{flalign}

For any ${\mu _i} \in {\mathbb{R}^{{m_i}}}$ and $a > 0$, it follows from \eqref{lemma4-proof-eq9} that
\begin{flalign}
\nonumber
& \;\;\;\;\;\sum\limits_{t = 1}^T {\sum\limits_{i = 1}^n {\| {{\mu _i}} \|\| {{{\hat x}_{i,t}} - {{\bar x}_t}} \|} } \\
\nonumber
& \le {\varpi _2}\sum\limits_{i = 1}^n {\| {{\mu _i}} \|} \\
\nonumber
&\;\; + \frac{1}{n}\sum\limits_{t = 2}^T {\sum\limits_{i = 1}^n {\sum\limits_{j = 1}^n {( {\frac{1}{{4a{F_2}{\alpha _t}}}{{\| {\varepsilon _{i,t - 1}^x} \|}^2} + a{F_2}{\alpha _t}{{\| {{\mu _j}} \|}^2}} )} } } \\
\nonumber
&\;\; + \sum\limits_{t = 2}^T {\sum\limits_{i = 1}^n {( {\frac{1}{{4a{F_2}{\alpha _t}}}{{\| {\varepsilon _{i,t - 1}^x} \|}^2} + a{F_2}{\alpha _t}{{\| {{\mu _i}} \|}^2}} )} } \\
\nonumber
&\;\; + \sum\limits_{t = 2}^T {\sum\limits_{i = 1}^n {\sum\limits_{j = 1}^n {( {\frac{1}{{2an{F_2}{\alpha _t}}}{{\| {\varepsilon _{i,t - 1}^x} \|}^2} + \frac{{an{F_2}{\tau ^2}{\alpha _t}}}{{2{{( {1 - \lambda } )}^2}}}{{\| {{\mu _j}} \|}^2}} )} } } \\
\nonumber
&\;\; + \sum\limits_{t = 1}^T {\sum\limits_{i = 1}^n {2{\tau _t}\| {{\mu _i}} \|} }  + \frac{{n\tau }}{{1 - \lambda }}\sum\limits_{t = 2}^T {\sum\limits_{i = 1}^n {{\tau _t}\| {{\mu _i}} \|} } \\
\nonumber
& \le {\varpi _2}\sum\limits_{i = 1}^n {\| {{\mu _i}} \|} + \sum\limits_{t = 2}^T {\sum\limits_{i = 1}^n {\frac{1}{{a{F_2}{\alpha _t}}}{{\| {\varepsilon _{i,t - 1}^x} \|}^2}} } \\
&\;\; + \sum\limits_{t = 2}^T {\sum\limits_{i = 1}^n {a{\varpi _3}{\alpha _t}{{\| {{\mu _i}} \|}^2}} } + {\varpi _4}\sum\limits_{t = 1}^T {\sum\limits_{i = 1}^n {{\tau _t}\| {{\mu _i}} \|} }. \label{lemma4-proof-eq10}
\end{flalign}
For any ${\mu _i} \in {\mathbb{R}^{{m_i}}}$ and $a > 0$, we have
\begin{flalign}
\nonumber
& \;\;\;\;\;\| {{\mu _i}} \|\| {{x_{i,t}} - {x_{i,t + 1}}} \| \\
\nonumber
& \le \| {{\mu _i}} \|\| {{x_{i,t}} - {z_{i,t + 1}}} \| + \| {{\mu _i}} \|\| {{z_{i,t + 1}} - {x_{i,t + 1}}} \| \\
& \le \| {{\mu _i}} \|\| {{x_{i,t}} - {z_{i,t + 1}}} \|
+ \frac{1}{{a{F_2}{\alpha _{t + 1}}}}{\| {\varepsilon _{i,t}^x} \|^2} + \frac{{a{F_2}{\alpha _{t + 1}}}}{4}{\| {{\mu _i}} \|^2}. \label{lemma4-proof-eq11}
\end{flalign}
From \eqref{Algorithm1-eq1} and ${\sum\nolimits_{i = 1}^n {[ {{W_t}} ]} _{ij}} = {\sum\nolimits_{j = 1}^n {[ {{W_t}} ]} _{ij}} = 1$, we have
\begin{flalign}
\nonumber
& \;\;\;\;\;\sum\limits_{i = 1}^n {\| {{x_{i,t}} - {z_{i,t + 1}}} \|} \\
\nonumber
& \le \sum\limits_{i = 1}^n {( {\| {{x_{i,t}} - {{\bar x}_t}} \| + \| {{{\bar x}_t} - {z_{i,t + 1}}} \|} )} \\
\nonumber
& \le \sum\limits_{i = 1}^n {( {\| {{x_{i,t}} - {{\hat x}_{i,t}} + {{\hat x}_{i,t}} - {{\bar x}_t}} \| + \| {{{\bar x}_t} - \sum\limits_{j = 1}^n {{{[ {{W_t}} ]}_{ij}}{{\hat x}_{j,t}}} } \|} )} \\
\nonumber
& \le \sum\limits_{i = 1}^n {( {\| {{{\hat x}_{i,t}} - {x_{i,t}}} \| + \| {{{\hat x}_{i,t}} - {{\bar x}_t}} \|} )} + \sum\limits_{i = 1}^n {\sum\limits_{j = 1}^n {{{[ {{W_t}} ]}_{ij}}\| {{{\bar x}_t} - {{\hat x}_{j,t}} } \|} } \\
& \le 2\sum\limits_{i = 1}^n {\| {{{\hat x}_{i,t}} - {{\bar x}_t}} \|}  + \sum\limits_{i = 1}^n {\tau _t}. \label{lemma4-proof-eq12}
\end{flalign}
From \eqref{lemma4-proof-eq10}--\eqref{lemma4-proof-eq12}, for any ${\mu _i} \in {\mathbb{R}^{{m_i}}}$ and $a > 0$, we have
\begin{flalign}
\nonumber
& \;\;\;\;\;\sum\limits_{t = 1}^T {\sum\limits_{i = 1}^n {{F_2}\| {{\mu _i}} \|\| {{x_{i,t}} - {x_{i,t + 1}}} \|} } \\
\nonumber
& \le \sum\limits_{t = 1}^T {\sum\limits_{i = 1}^n {2{F_2}\| {{\mu _i}} \|\| {{\hat{x}_{i,t}} - {{\bar x}_t}} \|} } + {F_2}\sum\limits_{t = 1}^T {\sum\limits_{i = 1}^n {{\tau _t}\| {{\mu _i}} \|} } \\
\nonumber
&\;\;  + \sum\limits_{t = 1}^T {\sum\limits_{i = 1}^n {\frac{1}{{a{\alpha _{t + 1}}}}{{\| {\varepsilon _{i,t}^x} \|}^2}} }  + \sum\limits_{t = 1}^T {\sum\limits_{i = 1}^n {\frac{{aF_2^2{\alpha _{t + 1}}}}{4}{{\| {{\mu _i}} \|}^2}} } 
\end{flalign}
\begin{flalign}
\nonumber
& \le 2{F_2}{\varpi _2}\sum\limits_{i = 1}^n {\| {{\mu _i}} \|}  + ( {2{\varpi _4} + 1} ){F_2}\sum\limits_{t = 1}^T {\sum\limits_{i = 1}^n {{\tau _t}\| {{\mu _i}} \|} } \\
& \;\;+ \sum\limits_{t = 1}^T {\sum\limits_{i = 1}^n {a{\varpi _5}{\alpha _{t + 1}}{{\| {{\mu _i}} \|}^2}} } + \sum\limits_{t = 1}^T {\sum\limits_{i = 1}^n {\frac{3}{{a{\alpha _{t + 1}}}}{{\| {\varepsilon _{i,t}^x} \|}^2}} }. \label{lemma4-proof-eq13}
\end{flalign}

Choosing $\| {{\mu _i}} \| = 1$ in \eqref{lemma4-proof-eq10} and \eqref{lemma4-proof-eq13} yields
\begin{flalign}
\nonumber
& \;\;\;\;\;\sum\limits_{t = 1}^T {\sum\limits_{i = 1}^n {2{F_2}\| {{{\hat x}_{i,t}} - {{\bar x}_t}} \|} } \\
\nonumber
& \le 2n{F_2}{\varpi _2} + 2n{F_2}{\varpi _4}\sum\limits_{t = 1}^T {{\tau _t}} \\
&  \;\;+ \sum\limits_{t = 2}^T {\sum\limits_{i = 1}^n {( {2a{F_2}{\varpi _3}{\alpha _t} + \frac{2}{{a{\alpha _t}}}{{\| {\varepsilon _{i,t - 1}^x} \|}^2}} )} }, \label{lemma4-proof-eq14}
\end{flalign}
and
\begin{flalign}
\nonumber
& \;\;\;\;\;\sum\limits_{t = 1}^T {\sum\limits_{i = 1}^n {{F_2}\left\| {{x_{i,t}} - {x_{i,t + 1}}} \right\|} } \\
\nonumber
& \le 2n{F_2}{\varpi _2} + ( {2{\varpi _4} + 1} )n{F_2}\sum\limits_{t = 1}^T {{\tau _t}} + \sum\limits_{t = 1}^T {an{\varpi _5}{\alpha _{t + 1}}} \\
&  \;\;+ \sum\limits_{t = 1}^T {\sum\limits_{i = 1}^n {\frac{3}{{a{\alpha _{t + 1}}}}{{\left\| {\varepsilon _{i,t}^x} \right\|}^2}} }. \label{lemma4-proof-eq15}
\end{flalign}

From \eqref{lemma4-proof-eq14} and \eqref{lemma4-proof-eq15}, and choosing $a = 10$ yields
\begin{flalign}
\nonumber
& \;\;\;\;\;\sum\limits_{t = 1}^T {\sum\limits_{i = 1}^n {{{\tilde \Delta }_{i,t + 1}}( {{\mathbf{0}_{{m_i}}}} )} } \\
\nonumber
& \le 4n{F_2}{\varpi _2} + ( {4{\varpi _4} + 1} )n{F_2}\sum\limits_{t = 1}^T {{\tau _t}} + \sum\limits_{t = 2}^T {20n{F_2}{\varpi _3}{\alpha _t}}
\\
\nonumber
& \;\;+ \sum\limits_{t = 1}^T {10n{\varpi _5}{\alpha _{t + 1}}} + \sum\limits_{t = 1}^T {\sum\limits_{i = 1}^n {\frac{3}{{10{\alpha _{t + 1}}}}} } {\| {\varepsilon _{i,t}^x} \|^2} \\
\nonumber
& \;\;+ \sum\limits_{t = 2}^T {\sum\limits_{i = 1}^n {\frac{2}{{10{\alpha _t}}}} } {\| {\varepsilon _{i,t - 1}^x} \|^2} - \sum\limits_{t = 1}^T {\sum\limits_{i = 1}^n {\frac{1}{{2{\alpha _{t + 1}}}}} } {\| {\varepsilon _{i,t}^x} \|^2} \\
& \le 4n{F_2}{\varpi _2} + ( {4{\varpi _4} + 1} )n{F_2}\sum\limits_{t = 1}^T {{\tau _t}}  + \sum\limits_{t = 1}^T {10n{\varpi _6}{\alpha _t}}. \label{lemma4-proof-eq16}
\end{flalign}

Combining \eqref{lemma4-proof-eq4}, \eqref{lemma4-proof-eq6} and \eqref{lemma4-proof-eq16} yields \eqref{lemma4-eq1}.

($\mathbf{ii}$)
We first provide a loose bound for network cumulative constraint violation.

We have
\begin{flalign}
\nonumber
& \;\;\;\;\;\mu _i^T{[ {{g_{i,t}}( {{x_{i,t}}} )} ]_ + } \\
\nonumber
& = \mu _i^T{[ {{g_{i,t}}( {{x_{j,t}}} )} ]_ + } + \mu _i^T{[ {{g_{i,t}}( {{x_{i,t}}} )} ]_ + } - \mu _i^T{[ {{g_{i,t}}( {{x_{j,t}}} )} ]_ + } \\
\nonumber
& \ge \mu _i^T{[ {{g_{i,t}}( {{x_{j,t}}} )} ]_ + } - \| {{\mu _i}} \|\| {{{[ {{g_{i,t}}( {{x_{i,t}}} )} ]}_ + } - {{[ {{g_{i,t}}( {{x_{j,t}}} )} ]}_ + }} \| \\
\nonumber
& \ge \mu _i^T{[ {{g_{i,t}}( {{x_{j,t}}} )} ]_ + } - \| {{\mu _i}} \|\| {{g_{i,t}}( {{x_{i,t}}} ) - {g_{i,t}}( {{x_{j,t}}} )} \| \\
\nonumber
& \ge \mu _i^T{[ {{g_{i,t}}( {{x_{j,t}}} )} ]_ + } - {F_2}\| {{\mu _i}} \|\| {{x_{i,t}} - {x_{j,t}}} \| \\
\nonumber
& \ge \mu _i^T{[ {{g_{i,t}}( {{x_{j,t}}} )} ]_ + } \\
\nonumber
& \;\;- {F_2}\| {{\mu _i}} \|\| {{{\hat x}_{i,t}} - {x_{i,t}}} \| - {F_2}\| {{\mu _i}} \|\| {{{\hat x}_{i,t}} - {{\bar x}_t}} \| \\
& \;\;- {F_2}\| {{\mu _i}} \|\| {{{\hat x}_{j,t}} - {x_{j,t}}} \| - {F_2}\| {{\mu _i}} \|\| {{{\hat x}_{j,t}} - {{\bar x}_t}} \|, \label{lemma4-proof-eq17}
\end{flalign}
where the second inequality holds since projection operator is non-expansive, and the third inequality holds due to \eqref{lemma3-proof-eq1b}.

Combining \eqref{lemma4-proof-eq2} and \eqref{lemma4-proof-eq17}, setting ${y_t} = y$, and summing over $j \in [ n ]$ yields
\begin{flalign}
\nonumber
& \;\;\;\;\;\sum\limits_{i = 1}^n {\big( {{\Delta _{i,t + 1}}( {{\mu _i}} ) + \frac{1}{n}\sum\limits_{j = 1}^n {\mu _i^T{{[ {{g_{i,t}}( {{x_{j,t}}} )} ]}_ + }}  - \frac{1}{2}{\beta _{t + 1}}{{\| {{\mu _i}} \|}^2}} \big)} \\
\nonumber
& \;\;+ \sum\limits_{i = 1}^n {{f_t}( {{x_{i,t}}} )}  - n{f_t}( y ) \\
\nonumber
& \le n{\varpi _1}{\gamma _{t + 1}} + \sum\limits_{i = 1}^n {{{\hat \Delta }_{i,t + 1}}( {{\mu _i}} )}  + \frac{1}{n}{\check{\Delta} _t} + 2n{F_2}{\tau _t} + \frac{{2nR( \mathbb{X} ){\tau _{t + 1}}}}{{{\alpha _{t + 1}}}} \\
& \;\;+ \frac{1}{{2{\alpha _{t + 1}}}}\sum\limits_{i = 1}^n {( {{{\| {y - {z_{i,t + 1}}} \|}^2} - \| {y - {z_{i,t + 2}}} \|} )}, \label{lemma4-proof-eq18}
\end{flalign}
where
\begin{flalign}
\nonumber
{{\hat \Delta }_{i,t + 1}}( {{\mu _i}} ) &= {{\tilde \Delta }_{i,t + 1}}( {{\mu _i}} ) + {F_2}\| {{\mu _i}} \|\| {{{\hat x}_{i,t}} - {{\bar x}_t}} \| + {F_2}\| {{\mu _i}} \|{\tau _t}, \\
\nonumber
{\check{\Delta} _t} &= \sum\limits_{i = 1}^n {\sum\limits_{j = 1}^n {{F_2}\| {{\mu _i}} \|\| {{{\hat x}_{j,t}} - {{\bar x}_t}} \|} }  + \sum\limits_{i = 1}^n {n{F_2}\| {{\mu _i}} \|{\tau _t}}.
\end{flalign}

Combining \eqref{lemma4-proof-eq10}, \eqref{lemma4-proof-eq13}--\eqref{lemma4-proof-eq15}, and choosing $a = 20$ yields
\begin{flalign}
\nonumber
& \;\;\;\;\;\sum\limits_{t = 1}^T {\sum\limits_{i = 1}^n {{{\hat \Delta }_{i,t + 1}}( {{\mu _i}} )} } \\
\nonumber
& \le 4n{F_2}{\varpi _2} + 20n{\varpi _6}\sum\limits_{t = 1}^T {{\alpha _t}}  + ( {4{\varpi _4} + 1} )n{F_2}\sum\limits_{t = 1}^T {{\tau _t}} \\
\nonumber
& \;\;+ 3{F_2}{\varpi _2}\sum\limits_{i = 1}^n {\| {{\mu _i}} \|}  + 20( {{F_2}{\varpi _3} + {\varpi _5}} )\sum\limits_{t = 1}^T {\sum\limits_{i = 1}^n {{\alpha _t}{{\| {{\mu _i}} \|}^2}} } \\
& \;\;+ ( {3{\varpi _4} + 2} ){F_2}\sum\limits_{t = 1}^T {\sum\limits_{i = 1}^n {{\tau _t}\| {{\mu _i}} \|} }  - \sum\limits_{t = 1}^T {\sum\limits_{i = 1}^n {\frac{1}{{20{\alpha _{t + 1}}}}} } {\| {\varepsilon _{i,t}^x} \|^2}. \label{lemma4-proof-eq19}
\end{flalign}

From \eqref{lemma4-proof-eq9}, we have
\begin{flalign}
\nonumber
& \;\;\;\;\;\sum\limits_{t = 1}^T {\sum\limits_{i = 1}^n {\sum\limits_{j = 1}^n {\| {{\mu _j}} \|\| {{{\hat x}_{i,t}} - {{\bar x}_t}} \|} } } \\
\nonumber
& \le n{\varpi _2}\sum\limits_{j = 1}^n {\| {{\mu _j}} \|}  + 2\sum\limits_{t = 2}^T {\sum\limits_{i = 1}^n {\sum\limits_{j = 1}^n {\| {\varepsilon _{i,t - 1}^x} \|\| {{\mu _j}} \|} } } \\
\nonumber
&  \;\;+ \frac{{n\tau }}{{1 - \lambda }}\sum\limits_{t = 2}^T {\sum\limits_{i = 1}^n {\sum\limits_{j = 1}^n {\| {\varepsilon _{i,t - 1}^x} \|\| {{\mu _j}} \|} } }  + n{\varpi _4}\sum\limits_{t = 1}^T {\sum\limits_{j = 1}^n {{\tau _t}\| {{\mu _j}} \|} } \\
\nonumber
& \le n{\varpi _2}\sum\limits_{j = 1}^n {\| {{\mu _j}} \|} \\
\nonumber
&  \;\;+ \sum\limits_{t = 2}^T {\sum\limits_{i = 1}^n {\sum\limits_{j = 1}^n {( {\frac{1}{{2a{F_2}{\alpha _t}}}{{\| {\varepsilon _{i,t - 1}^x} \|}^2} + 2a{F_2}{\alpha _t}{{\| {{\mu _j}} \|}^2}} )} } } \\
\nonumber
&  \;\;+ \sum\limits_{t = 2}^T {\sum\limits_{i = 1}^n {\sum\limits_{j = 1}^n {( {\frac{1}{{2a{F_2}{\alpha _t}}}{{\| {\varepsilon _{i,t - 1}^x} \|}^2} + \frac{{a{n^2}{F_2}{\tau ^2}{\alpha _t}}}{{2{{( {1 - \lambda } )}^2}}}{{\| {{\mu _j}} \|}^2}} )} } } 
\end{flalign}
\begin{flalign}
\nonumber
& \;\;+ n{\varpi _4}\sum\limits_{t = 1}^T {\sum\limits_{j = 1}^n {{\tau _t}\| {{\mu _j}} \|} } \\
\nonumber
& = n{\varpi _2}\sum\limits_{i = 1}^n {\| {{\mu _i}} \|}  + \sum\limits_{t = 2}^T {\sum\limits_{i = 1}^n {an{\varpi _3}{\alpha _t}{{\| {{\mu _i}} \|}^2}} } \\
& \;\;+ \sum\limits_{t = 2}^T {\sum\limits_{i = 1}^n {\frac{n}{{a{F_2}{\alpha _t}}}} } {\| {\varepsilon _{i,t - 1}^x} \|^2} + n{\varpi _4}\sum\limits_{t = 1}^T {\sum\limits_{i = 1}^n {{\tau _t}\| {{\mu _i}} \|} }. \label{lemma4-proof-eq20}
\end{flalign}

Choosing $a = 20$ in \eqref{lemma4-proof-eq20} yields
\begin{flalign}
\nonumber
&\;\;\;\;\;\frac{1}{n}\sum\limits_{t = 1}^T {{\check{\Delta} _t}} \\
\nonumber
&\le {F_2}{\varpi _2}\sum\limits_{i = 1}^n {\| {{\mu _i}} \|}  + \sum\limits_{t = 1}^T {\sum\limits_{i = 1}^n {20{F_2}{\varpi _3}{\alpha _t}{{\| {{\mu _i}} \|}^2}} } \\
&  \;\;+ \sum\limits_{t = 2}^T {\sum\limits_{i = 1}^n {\frac{1}{{20{\alpha _t}}}{{\| {\varepsilon _{i,t - 1}^x} \|}^2}} }  + {F_2}{\varpi _8}\sum\limits_{t = 1}^T {\sum\limits_{i = 1}^n {{\tau _t}\| {{\mu _i}} \|} }. \label{lemma4-proof-eq21}
\end{flalign}

Let ${h_{ij}}:\mathbb{R}_ + ^{{m_i}} \to \mathbb{R}$ be a function defined as
\begin{flalign}
\nonumber
{h_{ij}}( {{\mu _i}} ) &= \mu _i^T\sum\limits_{t = 1}^T {{{[ {{g_{i,t}}( {{x_{j,t}}} )} ]}_ + }} \\
& \;\;- \frac{1}{2}{\| {{\mu _i}} \|^2}\big( {\frac{1}{{{\gamma _1}}} + \sum\limits_{t = 1}^T {( {{\beta _t} + {\varpi _9}{\alpha _t}} )} } \big). \label{lemma4-proof-eq22}
\end{flalign}

From \eqref{lemma4-proof-eq3}, \eqref{lemma4-proof-eq5}, \eqref{lemma4-proof-eq19}, \eqref{lemma4-proof-eq21}, and \eqref{lemma4-proof-eq22}, summing \eqref{lemma4-proof-eq18} over $t \in [ T ]$ gives
\begin{flalign}
\nonumber
& \;\;\;\;\;\frac{1}{2}\sum\limits_{t = 1}^T {\sum\limits_{i = 1}^n {( {\frac{1}{{{\gamma _t}}} - \frac{1}{{{\gamma _{t + 1}}}} + {\beta _{t + 1}}} ){{\| {{q_{i,t}} - {\mu _i}} \|}^2}} } \\
\nonumber
&  \;\;+ \frac{1}{n}\sum\limits_{i = 1}^n {\sum\limits_{j = 1}^n {{h_{ij}}( {{\mu _i}} )} } + n{{\rm{Net}\mbox{-}\rm{Reg}}( {\{ {{x_{i,t}}} \},{y_{[T]}}} )} \\
\nonumber
&  \le 4n{F_2}{\varpi _2} + n{\varpi _1}\sum\limits_{t = 1}^T {{\gamma _t}}  + 20n{\varpi _6}\sum\limits_{t = 1}^T {{\alpha _t}}  + n{\varpi _{10}}\sum\limits_{t = 1}^T {{\tau _t}} \\
\nonumber
&  \;\;+ 4{F_2}{\varpi _2}\sum\limits_{i = 1}^n {\| {{\mu _i}} \|}  + {\varpi _{10}}\sum\limits_{t = 1}^T {\sum\limits_{i = 1}^n {{\tau _t}\| {{\mu _i}} \|} } \;\;\;\;\;\;\;\;\;\;\;\;\;\;\\
&  \;\;+ 2nR( \mathbb{X} )\sum\limits_{t = 1}^T {\frac{{{\tau _{t + 1}}}}{{{\alpha _{t + 1}}}}}  + \frac{{2nR{{( \mathbb{X} )}^2}}}{{{\alpha _{T + 1}}}}. \label{lemma4-proof-eq23}
\end{flalign}

Substituting ${\mu _i} = \mu _{ij}^0$ into \eqref{lemma4-proof-eq22} yields
\begin{flalign}
{h_{ij}}\left( {\mu _{ij}^0} \right) = \frac{{{{\| {\sum\nolimits_{t = 1}^T {{{[ {{g_{i,t}}( {{x_{j,t}}} )} ]}_ + }} } \|}^2}}}{{2\big( {\frac{1}{{{\gamma _1}}} + \sum\nolimits_{t = 1}^T {( {{\beta _t} + {\varpi _9}{\alpha _t}} )} } \big)}}. \label{lemma4-proof-eq24}
\end{flalign}

From ${g_t}( x ) = {\rm{col}}\big( {{g_{1,t}}( x ), \cdot  \cdot  \cdot ,{g_{n,t}}( x )} \big)$, we have
\begin{flalign}
\sum\limits_{i = 1}^n {\sum\limits_{j = 1}^n {{{\| \sum\limits_{t = 1}^T {{{[ {{g_{i,t}}( {{x_{j,t}}} )} ]}_ + }} } \|^2}} }  = \sum\limits_{j = 1}^n {{\| {\sum\limits_{t = 1}^T {{{[ {{g_t}( {{x_{j,t}}} )} ]}_ + }} } \|^2}}. \label{lemma4-proof-eq25}
\end{flalign}

From \eqref{ass-eq2a}, we have
\begin{flalign}
-{{\rm{Net}\mbox{-}\rm{Reg}}( {\{ {{x_{i,t}}} \},{y_{[T]}}} )} \le {F_1}T. \label{lemma4-proof-eq26}
\end{flalign}

From \eqref{ass-eq2b}, we have
\begin{flalign}
\| {\mu _{ij}^0} \| \le \frac{{{F_1}T}}{{\frac{1}{{{\gamma _1}}} + \sum\nolimits_{t = 1}^T {( {{\beta _t} + {\varpi _9}{\alpha _t}} )} }}. \label{lemma4-proof-eq27}
\end{flalign}

Substituting ${\mu _i} = \mu _{ij}^0$ into \eqref{lemma4-proof-eq23}, combining \eqref{lemma4-proof-eq24}--\eqref{lemma4-proof-eq27} yields \eqref{lemma4-eq2}.
\end{proof}

\section{Proof of Theorem 1}
Based on Lemma~4, we are now ready to prove Theorem~1.

($\mathbf{i}$)
For any constant $a \in \left[ {0,1} \right)$ and $T \in {\mathbb{N}_ + }$, it holds that
\begin{flalign}
\sum\limits_{t = 1}^T {\frac{1}{{{t^a}}}}  \le 1 + \int\limits_1^T {\frac{1}{{{t^a}}}} dt = \frac{{{T^{1 - a}} - a}}{{1 - a}} \le \frac{{{T^{1 - a}}}}{{1 - a}}. \label{theorem1-proof-eq1}
\end{flalign}

Form \eqref{theorem1-proof-eq1}, we have
\begin{flalign}
\sum\limits_{t = 1}^T {\sqrt {\frac{{{\Psi _t}}}{t}} }  \le \sqrt {{\Psi _T}} \sum\limits_{t = 1}^T {\frac{1}{{\sqrt t }}}  \le 2\sqrt {T{\Psi _T}}. \label{theorem1-proof-eq2}
\end{flalign}

From Cauchy--Schwarz inequality, we have
\begin{flalign}
\sum\limits_{t = 1}^T {\frac{{{\tau _{t + 1}}}}{{\sqrt {\frac{{{\Psi _{t + 1}}}}{{t + 1}}} }}}  \le \sum\limits_{t = 1}^T {\sqrt {{\tau _{t + 1}}} }  \le \sum\limits_{t = 1}^T {\sqrt {{\tau _t}} }  \le \sqrt {T{\Psi _T}}. \label{theorem1-proof-eq3}
\end{flalign}

From \eqref{theorem1-eq1}, we have
\begin{flalign}
\frac{t}{{{t^\kappa }}} - \frac{{t + 1}}{{{{\left( {t + 1} \right)}^\kappa }}} + \frac{1}{{{(t+1)^\kappa }}} = \frac{{t}}{{{t^\kappa }}} - \frac{{t}}{{{{( {t + 1} )}^\kappa }}} > 0. \label{theorem1-proof-eq4}
\end{flalign}

Combining \eqref{theorem1-eq1}, \eqref{lemma4-eq1}, \eqref{theorem1-proof-eq1}--\eqref{theorem1-proof-eq4} yields
\begin{flalign}
\nonumber
& \;\;\;\;\;{{\rm{Net}\mbox{-}\rm{Reg}}( {\{ {{x_{i,t}}} \},{y_{[ T ]}}} )} \\
\nonumber
& \le 4{F_2}{\varpi _2} + \frac{{{\varpi _1}}}{\kappa }{T^\kappa } + 20{\varpi _6}\sqrt {T{\Psi _T}}  + {\varpi _7}{\Psi _T} \\
\nonumber
&\;\;  + 2R( \mathbb{X} )\sqrt {T{\Psi _T}}  + 2\sqrt 2 R{( \mathbb{X} )^2}\sqrt {\frac{T}{{{\Psi _T}}}} \\
&\;\;  + 2R( \mathbb{X} )\sqrt {\frac{T}{{{\Psi _T}}}} {P_T}, \label{theorem1-proof-eq5}
\end{flalign}
which gives \eqref{theorem1-eq2}.

($\mathbf{ii}$)
From Cauchy--Schwarz inequality, we have
\begin{flalign}
{\big( {\frac{1}{n}\sum\limits_{i = 1}^n {\| {\sum\limits_{t = 1}^T {{{[ {{g_t}( {{x_{i,t}}} )} ]}_ + }} } \|} } \big)^2}
\le \frac{1}{n}\sum\limits_{i = 1}^n {{\| {\sum\limits_{t = 1}^T {{{[ {{g_t}( {{x_{i,t}}} )} ]}_ + }} } \|^2}}. \label{theorem1-proof-eq6}
\end{flalign}
Combining \eqref{theorem1-eq1}, \eqref{lemma4-eq2}, \eqref{theorem1-proof-eq1}--\eqref{theorem1-proof-eq4} yields
\begin{flalign}
\nonumber
& \;\;\;\;\;\frac{1}{n}\sum\limits_{i = 1}^n {{\| {\sum\limits_{t = 1}^T {{{[ {{g_t}( {{x_{i,t}}} )} ]}_ + }} } \|^2}} \\
\nonumber
& \le 8n{F_1}{F_2}{\varpi _2}T + 2n{F_1}{\varpi _{10}}T{\Psi _T} \\
\nonumber
&\;\;  + 2n(1 + \frac{{{T^{1 - \kappa }}}}{{1 - \kappa }} + 2{\varpi _9}\sqrt {T{\Psi _T}} )\big({F_1}T + 4{F_2}{\varpi _2} \\
\nonumber
&\;\;  + \frac{{{\varpi _1}}}{\kappa }{T^\kappa } + 40{\varpi _6}\sqrt {T{\Psi _T}}  + {\varpi _{10}}{\Psi _T} \\
&\;\;  + 2R( \mathbb{X} )\sqrt {T{\Psi _T}}  + 2\sqrt 2 R{( \mathbb{X} )^2}\sqrt {\frac{T}{{{\Psi _T}}}} \big). \label{theorem1-proof-eq7}
\end{flalign}

Combining \eqref{theorem1-proof-eq6}, \eqref{theorem1-proof-eq7} and
\begin{flalign}
\nonumber
& \;\;\;\;\;\sum\limits_{t = 1}^T {\| {{{[ {{g_t}( {{x_{i,t}}} )} ]}_ + }} \|}  \le {\sum\limits_{t = 1}^T {\| {{{[ {{g_t}( {{x_{i,t}}} )} ]}_ + }} \|} _1} \\
&  = {\| {\sum\limits_{t = 1}^T {{{[ {{g_t}( {{x_{i,t}}} )} ]}_ + }} } \|_1} \le \sqrt m \| {\sum\limits_{t = 1}^T {{{[ {{g_t}( {{x_{i,t}}} )} ]}_ + }} } \| \label{theorem1-proof-eq8}
\end{flalign}
yields \eqref{theorem1-eq3}.

\section{Proof of Theorem 2}
For any $T \ge 3$, it holds that
\begin{flalign}
\sum\limits_{t = 1}^T {\frac{1}{t}}  \le 1 + \int\limits_1^T {\frac{1}{t}} dt \le 1 + \log ( T ) \le 2\log ( T ). \label{theorem2-proof-eq1}
\end{flalign}

For any constant $b > 1$ and $T \in {\mathbb{N}_ + }$, there exists a constant $M > 0$ such that
\begin{flalign}
\sum\limits_{t = 1}^T {\frac{1}{{{t^b}}}}  \le M. \label{theorem2-proof-eq2}
\end{flalign}

($\mathbf{i}$)
Combining \eqref{theorem2-eq1} with ${\theta _3} \in ( {{\theta _1},1} )$, \eqref{lemma4-eq1}, \eqref{theorem1-proof-eq1} and \eqref{theorem1-proof-eq4} yields \\
\begin{flalign}
\nonumber
& \;\;\;\;\;{{\rm{Net}\mbox{-}\rm{Reg}}( {\{ {{x_{i,t}}} \},{y_{[ T ]}}} )} \\
\nonumber
& \le 4{F_2}{\varpi _2} + \frac{{{\varpi _1}{T^{{\theta _2}}}}}{{{\theta _2}}} + \frac{{10{\varpi _6}{\alpha _0}{T^{1 - {\theta _1}}}}}{{1 - {\theta _1}}} + \frac{{4R( \mathbb{X} ){\tau _0}{T^{1 + {\theta _1} - {\theta _3}}}}}{{( {1 + {\theta _1} - {\theta _3}} ){\alpha _0}}}  \\
&\;\; + \frac{{{\varpi _7}{\tau _0}{T^{1 - {\theta _3}}}}}{{1 - {\theta _3}}} + \frac{{4R{{( \mathbb{X} )}^2}{T^{{\theta _1}}}}}{{{\alpha _0}}} + \frac{{2R( \mathbb{X} ){T^{{\theta _1}}}{P_T}}}{{{\alpha _0}}}. \label{theorem2-proof-eq3}
\end{flalign}
Combining \eqref{theorem2-eq1} with ${\theta _3} = 1$, \eqref{lemma4-eq1}, \eqref{theorem1-proof-eq1}, \eqref{theorem1-proof-eq4} and \eqref{theorem2-proof-eq1} yields
\begin{flalign}
\nonumber
& \;\;\;\;\;{{\rm{Net}\mbox{-}\rm{Reg}}( {\{ {{x_{i,t}}} \},{y_{[ T ]}}} )} \\
\nonumber
& \le 4{F_2}{\varpi _2} + \frac{{{\varpi _1}{T^{{\theta _2}}}}}{{{\theta _2}}} + \frac{{10{\varpi _6}{\alpha _0}{T^{1 - {\theta _1}}}}}{{1 - {\theta _1}}} + \frac{{4R( \mathbb{X} ){\tau _0}{T^{{\theta _1}}}}}{{{\theta _1}{\alpha _0}}} \\
&\;\;  + 2{\varpi _7}{\tau _0}\log ( T ) + \frac{{4R{{( \mathbb{X} )}^2}{T^{{\theta _1}}}}}{{{\alpha _0}}} + \frac{{2R( \mathbb{X} ){T^{{\theta _1}}}{P_T}}}{{{\alpha _0}}}. \label{theorem2-proof-eq4}
\end{flalign}
Combining \eqref{theorem2-eq1} with $1 < {\theta _3} < 1 + {\theta _1}$, \eqref{lemma4-eq1}, \eqref{theorem1-proof-eq1}, \eqref{theorem1-proof-eq4}, and \eqref{theorem2-proof-eq2} yields
\begin{flalign}
\nonumber
& \;\;\;\;\;{{\rm{Net}\mbox{-}\rm{Reg}}( {\{ {{x_{i,t}}} \},{y_{[ T ]}}} )} \\
\nonumber
& \le 4{F_2}{\varpi _2} + \frac{{{\varpi _1}{T^{{\theta _2}}}}}{{{\theta _2}}} + \frac{{10{\varpi _6}{\alpha _0}{T^{1 - {\theta _1}}}}}{{1 - {\theta _1}}} + \frac{{4R( \mathbb{X} ){\tau _0}{T^{1 + {\theta _1} - {\theta _3}}}}}{{\left( {1 + {\theta _1} - {\theta _3}} \right){\alpha _0}}} \\
&\;\; + {\varpi _7}{\tau _0}M + \frac{{4R{{( \mathbb{X} )}^2}{T^{{\theta _1}}}}}{{{\alpha _0}}} + \frac{{2R( \mathbb{X} ){T^{{\theta _1}}}{P_T}}}{{{\alpha _0}}}. \label{theorem2-proof-eq5}
\end{flalign}
Combining \eqref{theorem2-eq1} with ${\theta _3} = 1 + {\theta _1}$, \eqref{lemma4-eq1}, \eqref{theorem1-proof-eq1}, \eqref{theorem1-proof-eq4}, \eqref{theorem2-proof-eq1} and \eqref{theorem2-proof-eq2} yields
\begin{flalign}
\nonumber
& \;\;\;\;\;{{\rm{Net}\mbox{-}\rm{Reg}}( {\{ {{x_{i,t}}} \},{y_{[ T ]}}} )} \\
\nonumber
& \le 4{F_2}{\varpi _2} + \frac{{{\varpi _1}{T^{{\theta _2}}}}}{{{\theta _2}}} + \frac{{10{\varpi _6}{\alpha _0}{T^{1 - {\theta _1}}}}}{{1 - {\theta _1}}} + \frac{{4R( \mathbb{X} ){\tau _0}\log ( T )}}{{{\alpha _0}}} \\
\nonumber
&\;\; + \frac{{4R( \mathbb{X} ){\tau _0}\log 2}}{{{\alpha _0}}} + {\varpi _7}{\tau _0}M + \frac{{4R{{( \mathbb{X} )}^2}{T^{{\theta _1}}}}}{{{\alpha _0}}} \\
&\;\; + \frac{{2R( \mathbb{X} ){T^{{\theta _1}}}{P_T}}}{{{\alpha _0}}}. \label{theorem2-proof-eq6}
\end{flalign}
Combining \eqref{theorem2-eq1} with ${\theta _3} > 1 + {\theta _1}$, \eqref{lemma4-eq1}, \eqref{theorem1-proof-eq1}, \eqref{theorem1-proof-eq4}, and \eqref{theorem2-proof-eq2} yields
\begin{flalign}
\nonumber
& \;\;\;\;\;{{\rm{Net}\mbox{-}\rm{Reg}}( {\{ {{x_{i,t}}} \},{y_{[ T ]}}} )} \\
\nonumber
& \le 4{F_2}{\varpi _2} + \frac{{{\varpi _1}{T^{{\theta _2}}}}}{{{\theta _2}}} + \frac{{10{\varpi _6}{\alpha _0}{T^{1 - {\theta _1}}}}}{{1 - {\theta _1}}} + \frac{{2R( \mathbb{X} ){\tau _0}M}}{{{\alpha _0}}} \\
&\;\; + {\varpi _7}{\tau _0}M + \frac{{4R{{( \mathbb{X} )}^2}{T^{{\theta _1}}}}}{{{\alpha _0}}} + \frac{{2R( \mathbb{X} ){T^{{\theta _1}}}{P_T}}}{{{\alpha _0}}}. \label{theorem2-proof-eq7}
\end{flalign}

From \eqref{theorem2-proof-eq3}--\eqref{theorem2-proof-eq7}, we have \eqref{theorem2-eq2}.

($\mathbf{ii}$)
Combining \eqref{theorem2-eq1} with ${\theta _3} \in ( {{\theta _1},1} )$, \eqref{lemma4-eq2}, \eqref{theorem1-proof-eq1} and \eqref{theorem1-proof-eq4} yields
\begin{flalign}
\nonumber
& \;\;\;\;\;\frac{1}{n}\sum\limits_{i = 1}^n {{\| {\sum\limits_{t = 1}^T {{{[ {{g_t}( {{x_{i,t}}} )} ]}_ + }} } \|^2}} \\
\nonumber
& \le 8n{F_1}{F_2}{\varpi _2}T + \frac{{2n{F_1}{\varpi _{10}}{\tau _0}{T^{2 - {\theta _3}}}}}{{1 - {\theta _3}}} \\
\nonumber
&  \;\; + 2n( {1 + \frac{{{T^{1 - {\theta _2}}}}}{{1 - {\theta _2}}} + \frac{{{\varpi _9}{\alpha _0}{T^{1 - {\theta _1}}}}}{{1 - {\theta _1}}}} )( {F_1}T + 4{F_2}{\varpi _2} \\
\nonumber
&  \;\; + \frac{{{\varpi _1}{T^{{\theta _2}}}}}{{{\theta _2}}} + \frac{{20{\varpi _6}{\alpha _0}{T^{1 - {\theta _1}}}}}{{1 - {\theta _1}}} + \frac{{{\varpi _{10}}{\tau _0}{T^{1 - {\theta _3}}}}}{{1 - {\theta _3}}} \\
&  \;\; + \frac{{4R( \mathbb{X} ){\tau _0}{T^{1 + {\theta _1} - {\theta _3}}}}}{{( {1 + {\theta _1} - {\theta _3}} ){\alpha _0}}} + \frac{{4R{{( \mathbb{X} )}^2}{T^{{\theta _1}}}}}{{{\alpha _0}}}). \label{theorem2-proof-eq8}
\end{flalign}
Combining \eqref{theorem2-eq1} with ${\theta _3} = 1$, \eqref{lemma4-eq2}, \eqref{theorem1-proof-eq1}, \eqref{theorem1-proof-eq4} and \eqref{theorem2-proof-eq1} yields
\begin{flalign}
\nonumber
& \;\;\;\;\;\frac{1}{n}\sum\limits_{i = 1}^n {{\| {\sum\limits_{t = 1}^T {{{[ {{g_t}( {{x_{i,t}}} )} ]}_ + }} } \|^2}} \\
\nonumber
& \le 8n{F_1}{F_2}{\varpi _2}T + 4n{F_1}{\varpi _{10}}{\tau _0}T\log ( T ) \\
\nonumber
&  \;\; + 2n( {1 + \frac{{{T^{1 - {\theta _2}}}}}{{1 - {\theta _2}}} + \frac{{{\varpi _9}{\alpha _0}{T^{1 - {\theta _1}}}}}{{1 - {\theta _1}}}})( {F_1}T + 4{F_2}{\varpi _2} \\
\nonumber
&  \;\; + \frac{{{\varpi _1}{T^{{\theta _2}}}}}{{{\theta _2}}} + \frac{{20{\varpi _6}{\alpha _0}{T^{1 - {\theta _1}}}}}{{1 - {\theta _1}}} + 2{\varpi _{10}}{\tau _0}\log ( T ) \\
&  \;\; + \frac{{4R( \mathbb{X} ){\tau _0}{T^{{\theta _1}}}}}{{{\theta _1}{\alpha _0}}} + \frac{{4R{{( \mathbb{X} )}^2}{T^{{\theta _1}}}}}{{{\alpha _0}}} ). \label{theorem2-proof-eq9}
\end{flalign}
Combining \eqref{theorem2-eq1} with $1 < {\theta _3} < 1 + {\theta _1}$, \eqref{lemma4-eq2}, \eqref{theorem1-proof-eq1}, \eqref{theorem1-proof-eq4}, and \eqref{theorem2-proof-eq2} yields
\begin{flalign}
\nonumber
& \;\;\;\;\;\frac{1}{n}\sum\limits_{i = 1}^n {{\| {\sum\limits_{t = 1}^T {{{[ {{g_t}( {{x_{i,t}}} )} ]}_ + }} } \|^2}} \\
\nonumber
& \le 8n{F_1}{F_2}{\varpi _2}T + 2n{F_1}{\varpi _{10}}{\tau _0}MT \\
\nonumber
&  \;\; + 2n( {1 + \frac{{{T^{1 - {\theta _2}}}}}{{1 - {\theta _2}}} + \frac{{{\varpi _9}{\alpha _0}{T^{1 - {\theta _1}}}}}{{1 - {\theta _1}}}} )( {F_1}T + 4{F_2}{\varpi _2} \\
\nonumber
&  \;\; + \frac{{{\varpi _1}{T^{{\theta _2}}}}}{{{\theta _2}}} + \frac{{20{\varpi _6}{\alpha _0}{T^{1 - {\theta _1}}}}}{{1 - {\theta _1}}} + {\varpi _{10}}{\tau _0}M \\
&  \;\; + \frac{{4R( \mathbb{X} ){\tau _0}{T^{1 + {\theta _1} - {\theta _3}}}}}{{( {1 + {\theta _1} - {\theta _3}} ){\alpha _0}}} + \frac{{4R{{( \mathbb{X} )}^2}{T^{{\theta _1}}}}}{{{\alpha _0}}} ). \label{theorem2-proof-eq10}
\end{flalign}
Combining \eqref{theorem2-eq1} with ${\theta _3} = 1 + {\theta _1}$, \eqref{lemma4-eq2}, \eqref{theorem1-proof-eq1}, \eqref{theorem1-proof-eq4}, \eqref{theorem2-proof-eq1} and \eqref{theorem2-proof-eq2} yields
\begin{flalign}
\nonumber
& \;\;\;\;\;\frac{1}{n}\sum\limits_{i = 1}^n {{\| {\sum\limits_{t = 1}^T {{{[ {{g_t}( {{x_{i,t}}} )} ]}_ + }} } \|^2}} \\
\nonumber
& \le 8n{F_1}{F_2}{\varpi _2}T + 2n{F_1}{\varpi _{10}}{\tau _0}MT 
\end{flalign}
\begin{flalign}
\nonumber
&  \;\; + 2n( {1 + \frac{{{T^{1 - {\theta _2}}}}}{{1 - {\theta _2}}} + \frac{{{\varpi _9}{\alpha _0}{T^{1 - {\theta _1}}}}}{{1 - {\theta _1}}}} )( {F_1}T + 4{F_2}{\varpi _2} \\
\nonumber
&  \;\; + \frac{{{\varpi _1}{T^{{\theta _2}}}}}{{{\theta _2}}} + \frac{{20{\varpi _6}{\alpha _0}{T^{1 - {\theta _1}}}}}{{1 - {\theta _1}}} + {\varpi _{10}}{\tau _0}M + \frac{{4R( \mathbb{X} ){\tau _0}\log ( T )}}{{{\alpha _0}}}\\
&  \;\; + \frac{{4R( \mathbb{X} ){\tau _0}\log 2}}{{{\alpha _0}}} + \frac{{4R{{( \mathbb{X} )}^2}{T^{{\theta _1}}}}}{{{\alpha _0}}} ). \label{theorem2-proof-eq11}
\end{flalign}
Combining \eqref{theorem2-eq1} with ${\theta _3} > 1 + {\theta _1}$, \eqref{lemma4-eq2}, \eqref{theorem1-proof-eq1}, \eqref{theorem1-proof-eq4}, and \eqref{theorem2-proof-eq2} yields
\begin{flalign}
\nonumber
& \;\;\;\;\;\frac{1}{n}\sum\limits_{i = 1}^n {{\| {\sum\limits_{t = 1}^T {{{[ {{g_t}( {{x_{i,t}}} )} ]}_ + }} } \|^2}} \\
\nonumber
& \le 8n{F_1}{F_2}{\varpi _2}T + 2n{F_1}{\varpi _{10}}{\tau _0}MT \\
\nonumber
&  \;\; + 2n( {1 + \frac{{{T^{1 - {\theta _2}}}}}{{1 - {\theta _2}}} + \frac{{{\varpi _9}{\alpha _0}{T^{1 - {\theta _1}}}}}{{1 - {\theta _1}}}} )( {F_1}T + 4{F_2}{\varpi _2} + \frac{{{\varpi _1}{T^{{\theta _2}}}}}{{{\theta _2}}} \\
&  \;\;  + \frac{{20{\varpi _6}{\alpha _0}{T^{1 - {\theta _1}}}}}{{1 - {\theta _1}}} + {\varpi _{10}}{\tau _0}M
 + \frac{{2R( \mathbb{X} ){\tau _0}M}}{{{\alpha _0}}} + \frac{{4R{{( \mathbb{X} )}^2}{T^{{\theta _1}}}}}{{{\alpha _0}}} ). \label{theorem2-proof-eq12}
\end{flalign}

Combining \eqref{theorem1-proof-eq6}, \eqref{theorem1-proof-eq8} and \eqref{theorem2-proof-eq8}--\eqref{theorem2-proof-eq12} yields \eqref{theorem2-eq3}.


\end{document}